\documentclass[sn-mathphys-num]{sn-jnl}

\usepackage{graphicx}%
\usepackage{multirow}%
\usepackage{amsmath,amssymb,amsfonts}%
\usepackage{amsthm}%
\usepackage{mathrsfs}%
\usepackage[title]{appendix}%
\usepackage{xcolor}%
\usepackage{textcomp}%
\usepackage{manyfoot}%
\usepackage{booktabs}%
\usepackage{algorithm}%
\usepackage{algorithmicx}%
\usepackage{algpseudocode}%
\usepackage{listings}%
\usepackage[utf8]{inputenc}
\usepackage[T1]{fontenc}
\theoremstyle{thmstyleone}%
\newtheorem{theorem}{Theorem}

\theoremstyle{thmstyletwo}%
\newtheorem{lemma}{Lemma}
\theoremstyle{thmstylethree}%
\newtheorem{definition}{Definition}%

\begin{document}

\title[Article Title]{Extreme points, strongly extreme points and exposed points in Orlicz--Lorentz spaces}


\author[1]{\fnm{Di} \sur{Wang}}\email{wangd267@mail2.sysu.edu.cn}

\author*[1]{\fnm{Yongjin} \sur{Li}}\email{stslyj@mail.sysu.edu.cn}

\affil[1]{\orgdiv{School of Mathematics}, \orgname{Sun Yat-sen University}, \orgaddress{\city{Guangzhou}, \postcode{510275}, \country{China}}}




\abstract{

In this paper, we investigate the extremal structure of the unit ball in the most general classes of Orlicz--Lorentz spaces. the characterizations of extreme points, 
strongly extreme points, and exposed points are given for Orlicz--Lorentz function spaces $\Lambda_{\varphi,\omega}$ generated by an arbitrary Orlicz function $\varphi$ and a non--increasing weight function $\omega$, without assuming $\varphi$ is an $N$-function and $\omega$ is strict decreasing. 
Furthermore, we provide necessary and sufficient conditions for a functional in the dual space to attain its Luxemburg norm at $x \in \Lambda_{\varphi,\omega}$ without assuming that $\varphi$ is an $N$--function. The supporting functionals of $x \in \Lambda_{\varphi,\omega}$ are also characterized.
}




\maketitle

\section{Introduction}\label{sec1}
This article aims to characterize the extreme points, strongly extreme points and exposed points in Orlicz--Lorentz function spaces $\Lambda_{\varphi, \omega}^{o}$ generated by arbitrary Orlicz function $\varphi$ and decreasing weight $\omega$. 
It has been presented in \cite{kaminska1990extreme} for results about extreme points, but assuming that $\omega$ is strictly decreasing. In this article we will remove these restrictions. 
In our study of strongly extreme points, we have applied methods different from these in \cite{WangSEP2023}. 
In \cite{Wang2024}, the descriptions for supporting functional are given assuming Orlicz function $\varphi$ is $N$-function. In this article we remove this assumption, and give the description of the supporting functional when Orlicz function is not an $N$-function.

Suppose $C$ is a non--empty convex subset of Banach space $X$. A point $x\in C$ is called an extreme point of $C$ if $2x=y+z$, $y,z\in C$ implies $y=z$.
Moreover, if $2x=y_{n}+z_{n}$ and $d(y_{n}, C)\rightarrow 0$, $d(z_{n}, C)\rightarrow 0$ imply $\|y_{n}-z_{n}\|\rightarrow\infty$ as $n\rightarrow\infty$, then $x$ is called a strongly extreme point.

Let $Grad(x)$ denote the set of supporting functionals at $x$, $Grad(x)=\{f\in S(X^{*}): f(x)=\|x\|\}$ where $S(X^{*})$ denote the unit sphere of $X^{*}$ and $X^{*}$ denote the dual space of $X$.  $x\in X$ is said to be an exposed point of $C$ if there exists $f\in Grad(x)$ such that $f(x)>f(y)$ for all $y\in C$ and $y\neq x$.
In this case, the functional $f$ is called an exposed functional of $C$ and exposing $C$ at $x$.

For $f\in L_{0}$, its non-increasing rearrangement $f^*$ is defined by
$$
f^*(t):=\inf\{\lambda >0 : \mu_{f}(\lambda)\leq t \}, \;  t>0,
$$
where $\mu_{f}$ is the distribution function of $f$, and it is defined by
$$
\mu_{f}(t):=\mu(\{s \in R^{+} : | f(s)|  > t\}), \;  t>0.
$$


A function $\sigma:(0,r)\rightarrow (0,r)$ $(r\leq \infty)$ is called a measure preserving transformation (\cite{Bennett1953,Brudnyi19882009}) if for every measurable set $E\subset (0,r)$, $\sigma^{-1}(E) $ is measurable and $\mu(\sigma^{-1}(E))=\mu(E)$.
It was given in \cite[Corollary 7.6]{Bennett1953} that for a resonant measure space $(R, \mu)$ and a non--negative $\mu$-measurable function $f$ on $R$ satisfying $\lim_{t\rightarrow\infty}f^{*}(t)=0$, there exists a measure-preserving transformation $\sigma$ from the support of $f$ onto the support of $f^{*}$ such that $f=f^{*}\circ\sigma\;\mu-a.e.$ on the support of $f$. More results about measure preserving transformation please refer to \cite{Bennett1953,Ryff1970449}.

For a function $\omega: (0,\infty) \rightarrow (0, \infty)$, It $\omega$ is non-increasing, locally integrable with respect to the Lebesgue measure $\mu$ and
$\int_{0}^{\infty}\omega(t)dt=\infty$, then $\omega$ is called a weight function.
The Lorentz space $\Lambda_{\omega}$ is defined as
\begin{equation*}
\Lambda_{\omega}=\left\{ f\in L_{0}: \|f\|_{\Lambda_{\omega}}=\int_{R^{+}}f^{*}(t)\omega(t)dt=\int_{R^{+}}f^{*}(t)dW<\infty \right\},
\end{equation*}
and the Marcinkiewicz space $M_{W}$ is defined as
\begin{equation*}
M_{W}=\left\{ f\in L^{0}: \|f\|_{M_{W}}=\sup_{\alpha>0}\frac{\int_{0}^{\alpha}f^{*}(t)dt}{W(\alpha)}<\infty \right\}.
\end{equation*}
where 
$
W(t):=\int_{0}^{t}\omega(s)ds, \; t\geq0.
$
The two spaces are K\"{o}the dual to each other.

For $f\in L_{0}$ and $g\in L_{0}$, if $\int_{0}^{t}f^{*}(s)ds<\int_{0}^{t}g^{*}(s)ds$ for $t>0$, then we say $f$ is submajorized by $g$, and we denote this by $f\prec g$.
Obviously $\|f\|_{M_{W}}\leq 1$ whenever $f\prec \omega$.

A function $\varphi:R^{+}\rightarrow R^{+}$ is an Orlicz function(\cite{Chen1996}) if $\varphi$ is convex, $\varphi(0)=0$ and $\varphi(t)>0$ for all $t>0$. For any Orlicz function $\varphi$, its complementary function $\psi$ in the sense of Young is defined by the formula
$$
\psi(t)=\sup_{s>0}\{| st | -\varphi(s)\}
$$
where $t>0$. Obviously, if $\varphi$ is an Orlicz function, then $\psi$ is an Orlicz function as well. For $p(s)$ and $q(t)$ we denote the right-derivative of $\varphi$ and $\psi$, respectively. Both $p(s)$ and $q(t)$ are non-decreasing and satisfy:
$$
\varphi(u)=\int_{0}^{u}p(t)dt, \quad \psi(v)=\int_{0}^{v}q(s)ds \quad u\geq0, v\geq0.
$$
$\varphi$ and $\psi$ satisfy Young Inequality:
$$
uv \leq \varphi(u)+ \psi(v), \quad u\geq0, v\geq0,
$$
and the equation $uv=\varphi(u)+\psi(v)$ hold whenever $u\in [q_{-}(v)sign\:v,\;q(v)sign\:v]$ or $v\in [p_{-}(u)sign\: u,\; p(u)sign\: u]$, where $p_{-}(t)=sup\{p(s): 0\leq s <t\}$, $q_{-}(t)=sup\{q(s): 0\leq s <t\}$, and $p_{-}(0)=q_{-}(0)=0$. 
For Orlicz function $\varphi$, if $\varphi(t)$ is affine on the interval $[a,b]$, and it is not affine on either $[a-\varepsilon, b]$ or $[a, b+\varepsilon]$ for any $\varepsilon>0$. The interval $[a,b]$ is called an affine interval of $\varphi$, denoted by $AI[a,b]$.
Define

\begin{align*}
	A^{\prime}&=\bigcup\{a_{i}: AI(a_{i}, b_{i})\: \text{with}\: p_{-}(a_{i})=p(a_{i})\}.\\
	B^{\prime}&=\bigcup\{b_{i}: AI(a_{i}, b_{i})\:\text{with}\: p_{-}(b_{i})=p(b_{i})\}.\\
	       A\:&=\bigcup\{a_{i}: AI(a_{i}, b_{i})\:\text{with}\: p_{-}(a_{i})<p(a_{i})\}.\\
	       B\:&=\bigcup\{b_{i}: AI(a_{i}, b_{i})\:\text{with}\: p_{-}(b_{i})<p(b_{i})\}.
\end{align*}
For an Orlicz function $\varphi$, by $S$ we denoted the strict convex point of $\varphi$. Let $\{[a_{i}, b_{i}]\}_{i\in N}$ be the set of all AIs of $\varphi$. Obviously, $S=R^{+}\backslash \cup_{i\in N}[a_{i}, b_{i}]$.
By $S^{\prime}$ we denoted $R^{+}\backslash S$.
We recall that an Orlicz function $\varphi$ satisfies the $\Delta_{2}$-condition($\varphi\in\Delta_{2}$) if there exist $C>0$ such that $\varphi(2t)\leq C\varphi(t)$ for all $t>0$. We further say $\varphi$ is an $N$-function whenever:
$$
\lim_{t\rightarrow 0+}\frac{\varphi(t)}{t}=0  \quad and \quad \lim_{t\rightarrow \infty}\frac{\varphi(t)}{t}=\infty.
$$

If $\omega(t)$ is a constant in an interval $A$, and for any interval $B$ such that $ A\subsetneqq B$, $\omega(t)$ is not a constant when $t\in B$. Then $A$ is called a maximal constant interval of $\omega$. $L(\omega)$ is the set composed of the maximal constant intervals.

For decreasing weight function $\omega$,
let $W(a,b)=\int_{a}^{b}\omega(t)dt$.
For arbitrary $f\in  L_{0}$, let $R(a_{1}, b_{1})=\frac{F(a_{1}, b_{1})}{W(a_{1}, b_{1})}=\frac{\int_{a_{1}}^{b_{1}}f(t)dt}{\int_{a_{1}}^{b_{1}}\omega(t)dt}$.

\begin{definition}[\cite{Halperin195305}]

if for every $s\in (a_{1}, b_{1})$, $W(a_{1}, s)=\int_{a_{1}}^{s}\omega(t)dt>0$ and $R(a_{1},s)\leq R(a_{1}, b_{1})$, then $(a_{1}, b_{1})$ is called a level interval of $f$ with respect to $\omega$. If the level interval is not contained in a larger one, it is called a maximal level interval.
\end{definition}
The maximal level intervals are non-overlapping and denumerable.
We will introduce the definition of level function and inverse level function.
\begin{definition}\cite[Definition 3.2]{Halperin195305}, \cite[Definition 4.2]{Kaminska2014229}
The level function of $f$ with respect to $\omega$ is denoted by $f^{0}$, and $f^{0}$ is defined by
\begin{equation*}
f^{0}(t)=\left\{
\begin{aligned}
&R(a_{n},b_{n})\omega(t) \;\;for\; t\in (a_{n}, b_{n}),\\
&\;\;\;f(t) \;\;\;\;\; ~~~~~~otherwise.
\end{aligned}
\right.
\end{equation*}
The inverse level function of $\omega$ with respect to $f$ is defined by
\begin{equation*}
\omega^{f}(t) =
\left\{
\begin{aligned}
&\frac{f(t)}{R(a_{n}, b_{n})}\;\;\; for \;t\in (a_{n}, b_{n}),\\
&\;\;\omega(t) \;\;\;\; ~~~~otherwise.
\end{aligned}
\right.
\end{equation*}
Where $(a_{n}, b_{n})$ is an enumeration of all maximal intervals of $f$ with respect to $\omega$, and $R(a_{n}, b_{n})=\frac{F(a_{n}, b_{n})}{W(a_{n}, b_{n})}$.
\end{definition}
Notice that for arbitrary$f\in\Lambda_{\varphi, \omega}$, we have $\omega^{f^{*}}=\omega$ since $f^{*}$ is decreasing.

For $f\in\Lambda_{\varphi, \omega}$, its modular $\rho_{\varphi, \omega}$ defined by
$$
\rho_{\varphi, \omega}(f)=\int_{0}^{\infty}\varphi(f^*(t))\omega(t)dt.
$$
The Orlicz-Lorentz function space is defined by
$$
\Lambda_{\varphi, \omega}=\{f\in L_{0}: \rho_{\varphi, \omega}(\lambda f)=\int_{0}^{\infty}\varphi(\lambda f^*(t))\omega(t)dt < \infty ~for~some~ \lambda >0\}.
$$
$\Lambda_{\varphi, \omega}$ becomes Banach space under the Luxemburg norm or the Orlicz norm(\cite{Kaminska199029,Foralewski2023}). The Luxemburg norm is defined by
$$
\|f\|=\|f\|_{\varphi, \omega}=\inf\{\lambda>0: \rho_{\varphi, \omega}(\frac{f}{\lambda})\leq 1\},
$$
and the Orlicz norm is defined by
$$
\|f\|^{o}=\|f\|_{\varphi, \omega}^{o}=\sup_{\rho_{\psi, \omega}(g)\leq 1}\int_{0}^{\infty}f^*(t)g^*(t)\omega(t)dt.
$$
Besieds, the Orlicz norm can be written in Amemiya type(\cite[Theorem 3.1]{Foralewski2023}), that is, 
\begin{equation}
\|f\|_{\varphi, \omega}^{o}=\inf_{k>0}{\frac{1}{k}\left(1+\rho_{\varphi, \omega}(kf)\right)}
\end{equation}
holds for every $f\in \Lambda_{\varphi, \omega}^{o}$.
$E_{\varphi, \omega}$ is a subspace of $\Lambda_{\varphi, \omega}$ denote the subspace of $\Lambda_{\varphi, \omega}$ such that
$$
E_{\varphi, \omega}=\{f\in L_{0}: \rho_{\varphi, \omega}(\lambda f)=\int_{0}^{\infty}\varphi(\lambda f^*(t))\omega(t)dt < \infty ~for~all~ \lambda >0\}.
$$
For arbitrary $f\in \Lambda_{\varphi, \omega}$,
$\theta(f)$ is defined as follow.
\begin{equation}
	\theta(f)=\inf\left\{\lambda>0: \rho_{\varphi, \omega}(\frac{f}{\lambda})<\infty\right\}.
\end{equation}
Now we will introduce the space $\mathcal{M}_{\varphi, \omega}$. For arbitrary Orlicz function $\varphi$ and decreasing weight $\omega$ the modular $P_{\varphi, \omega}$ is given by
\begin{equation*}
	P_{\varphi, \omega}(f)=\inf\left\{\int_{R^{+}}\varphi(\frac{f^{*}}{|g|});g\prec\omega	
		\right\},
\end{equation*}
and the space $\mathcal{M}_{\varphi, \omega}$ is defined by
\begin{equation*}
	\mathcal{M}_{\varphi,\omega}=\{
	f\in L_{0}: P_{\varphi, \omega}(\lambda f)<\infty\:for ~some~\lambda >0
	\}
\end{equation*}
Let $\mathcal{M}_{\varphi, \omega}=(M_{\varphi, \omega}, \|\cdot\|_{\varphi, \omega})$ and $\mathcal{M}_{\varphi, \omega}^{o}=(\mathcal{M}_{\varphi, \omega}^{o})$, where $\|\cdot\|_{\varphi, \omega}$ and $\|\cdot\|_{\mathcal{M}_{\varphi, \omega}^{o}}$ are defined by
\begin{align}
	\|f\|_{\mathcal{M}_{\psi, \omega}^{o}}
	&=\inf_{k>0}\left\{ \frac{1}{k}\left(P_{\psi, \omega}(kf)+1\right)\right\},\\
	\|f\|_{\mathcal{M}_{\psi, \omega}}
	&=\inf\left\{\lambda>0: \: P_{\psi, \omega}\left(\frac{f}{\lambda}\right)\leq 1\right\}.
\end{align}
In  Section \ref{Auxiliary}, we will introduced some basic properties of $\Lambda_{\varphi, \omega}$ and its dual, as well as conclusions about decreasing rearranged function. Section \ref{MainResult} is the main results of this article.
In Subsection \ref{SecExtremepoint}, the criteria of extreme points of Orlicz-Lorentz space generated by arbitrary Orlicz function and non-increasing rearrangement function are given.
When addressing this issue, we removed the restriction in \cite{kaminska1990extreme} that the weight function must be strictly decreasing.
Using this result, the criteria for strongly extreme points are given in Subsection \ref{SectionStronglyextremepoint} and we provided more methods for handling sequence of rearrangement function by utilizing measure-preserving transformations.
In Subsection \ref{Normexpression} we give the expression of the norm $\|v\|_{\mathcal{M}_{\varphi, \omega}^{o}}$ according to whether the set $K_{\mathcal{M}}(v)$ is an empty set. The necessity and sufficiency conditions are presented in Subsection \ref{Supportingfunctional} for a functional $f=L_{v}+s\in(\Lambda_{\varphi, \omega})^{*}$ whose norm is attained at $x\in\Lambda_{\varphi, \omega}$ without assuming $\varphi$ is an N-function. The supporting functional are also characterized. In Subsection \ref{Exposed}, the exposed points in Orlicz-Lorentz space are characterized.

\section{Auxiliary results}\label{Auxiliary}

Let $\Lambda_{\varphi, \omega}=(\Lambda_{\varphi, \omega}^{o}, \|\cdot\|_{\varphi, \omega})$, 
$\Lambda_{\varphi, \omega}^{o}=(\Lambda_{\varphi, \omega}, \|\cdot\|_{\varphi, \omega})$. 
\begin{lemma}\cite[Theorem 8.10]{2019abstractlorentz},\cite[Theorem 2.2]{Kaminska2014229}
For arbitrary Orlicz function $\varphi$ and decreasing weight $\omega$, 
(1) the k\"{o}the dual of Orlicz-Lorentz spaces $\Lambda_{\varphi, \omega}$ and $\Lambda_{\varphi, \omega}^{o}$ are expressed as
\begin{align*}
(\Lambda_{\varphi, \omega}, \|\cdot\|_{\varphi, \omega})^{\prime}
&=(\mathcal{M}_{\psi, \omega}, \|\cdot\|_{\mathcal{M}_{\psi, \omega}^{o}})\\
(\Lambda_{\varphi, \omega}, \|\cdot\|_{\varphi, \omega}^{o})^{\prime}
&=(\mathcal{M}_{\psi, \omega}, \|\cdot\|_{\mathcal{M}_{\varphi, \omega}})
\end{align*}
with equality of corresponding norms.\\
(2) If $\varphi\in\Delta_{2}$ and $\int_{0}^{\infty}\omega(t)dt = W(\infty)=\infty$. Then the dual spaces $(\Lambda_{\varphi, \omega})^{*}$ and $(\Lambda_{\varphi, \omega}^{o})^{*}$ are isometrically isomorphic to their corresponding K\"{o}the dual spaces. For functional $L_{v}\in (\Lambda_{\varphi, \omega})^{*}(resp., \Phi\in (\Lambda_{\varphi, \omega}^{o})^{*})$, there exists $v\in \mathcal{M}_{\psi, \omega}^{o}(resp., v\in \mathcal{M}_{\psi, \omega}^{o}$ such that
$$
\Phi(f)=L_{\phi}(f)=\int_{0}^{\infty}f(t)\phi(t)dt,\;\; f\in \Lambda_{\varphi, \omega}.
$$
and $\|\Phi\|_{(\Lambda_{\varphi, \omega})^{*}}=\|\phi\|_{\mathcal{M}_{\psi, \omega}}^{o}\;(resp., \|\Phi\|_{(\Lambda_{\varphi, \omega}^{o})^{*}}=\|\phi\|_{\mathcal{M}_{\psi, \omega}})$.
\end{lemma}

\begin{lemma} \cite[Theorem 2.5]{Kaminska199029},
For  $f\in\Lambda_{\varphi, \omega}$, the following condition are equivalent:\\
(1)$\varphi \in \Delta_{2}$,\\
(2)For arbitrary $\varepsilon>0$, there exists $\delta>0$, such that $\rho_{\varphi, \omega}(f)\geq 1-\varepsilon$ whenever $\|f\|_{\varphi, \omega}\geq 1-\delta$,\\
(3)$\rho_{\varphi, \omega}(f)=1$ if and only if $\|f\|_{\varphi, \omega}=1$,\\
(4)$\rho_{\varphi, \omega}(f_{n})\rightarrow 0$ if and only if $\|f_{n}\|_{\varphi, \omega}\rightarrow 0\:(n\rightarrow\infty)$.
\end{lemma}


For $x\in\Lambda_{\varphi, \omega}^{o}$, define
\begin{align*}
&k^{* }=k^{* }(x)=\inf\{k>0: \rho_{\psi, \omega}(p(kx))\geq 1\}.\\
&k^{**}=k^{**}(x)=\sup\{k>0: \rho_{\psi, \omega}(p(kx))\leq 1\}.\\
&K(x)=[K^{*}(x), k^{**}(x)]
\end{align*}
\begin{lemma}\cite[Theorem 3.1, Theorem 3.2]{Foralewski2023}
Let $\varphi$ be an Orlicz function and $\omega$ be a decreasing weight. $x\in\Lambda_{\varphi, \omega}^{o}$ and $K(x)\neq \emptyset$.
Then $x=\frac{1}{k}\left( 1+\rho_{\varphi, \omega}(kx) \right)$ if and only if $k\in K(x)$.
\end{lemma}
\begin{lemma}\cite[Theorem 2.8]{Wang2024}
For arbitrary $f\in \mathcal{M}_{\varphi, \omega}^{o}$,
if there exists $k>0$ such that
$$
\int_{R^{+}}\psi(p(\frac{kf^{*}(t)}{\omega^{f^{*}}}))\omega^{f^{*}}(t)=1.
$$
Then
$$
\|f\|_{\mathcal{M}_{\varphi, \omega}}^{o}=\frac{1}{k}\left(1+P_{\varphi, \omega}(kf)\right).
$$
\end{lemma}

For arbitrary $f\in \mathcal{M}_{\varphi, \omega}$, define
\begin{align*}
k^{*}_{\mathcal{M}} &=\inf\{k>0: \int_{0}^{\infty}\psi(p(\frac{kf^{*}(t)}{\omega^{f^{*}}(t)}))\omega^{f^{*}}(t)dt\geq 1\},\\
k^{**}_{\mathcal{M}}&=\sup\{k>0: \int_{0}^{\infty}\psi(p(\frac{kf^{*}(t)}{\omega^{f^{*}}(t)}))\omega^{f^{*}}(t)dt\leq 1\}.
\end{align*}
Obviously, $k^{*}_{\mathcal{M}}\leq k^{**}_{\mathcal{M}}$. Let $K_{\mathcal{M}}=[k_{\mathcal{M}}^{*}, k^{**}_{\mathcal{M}}]$. If $k^{**}_{\mathcal{M}}<\infty$,
Then $K_{\mathcal{M}}\neq \emptyset$.
We will discuss the norm expression $\|\cdot\|_{\mathcal{M}_{\varphi, \omega}^{o}}$ generated by arbitrary Orlicz function in Theorem \ref{Norminkothedual}.

\begin{lemma}\cite[Theorem 9, p.g. 36]{kan1982},
For arbitrary $f\in (\Lambda_{\varphi, \omega}^{o})^{*}$, f has a unique decomposition
$$
f=L_{g}+s,\;\;g\in \mathcal{M}_{\psi, \omega},\: s\in F
$$
where s is singular, $F$ is the set of singular functionals at $x$.
\end{lemma}

\begin{lemma}\cite[Theorem 2.48]{Chen1996}
For $x\in \Lambda_{\varphi, \omega}$ satisfying $\theta(x)\neq 0$,  there exist two singular functionals $s_{1}$ and $s_{2}$ such that  $s_{1}\neq s_{2}$ and $s_{1}(x)=s_{2}(x)=\theta(x)$.
\end{lemma}
For $f\in(\Lambda_{\varphi, \omega})^{*}$, let  $\|\cdot\|=\|\cdot\|_{(\Lambda_{\varphi, \omega}^{o})^{*}}$ and $\|\cdot\|^{o}=\|\cdot\|_{(\Lambda_{\varphi, \omega})^{*}}$.
\begin{lemma}\cite[Corollary 1.49]{Chen1996}
$\|f\|=\|f\|^{o}$ if and only if $f\in F$.
\end{lemma}

\begin{lemma}\cite[Lemma 2.22]{Wang2024}
	For arbitrary $f\in (\Lambda_{\varphi, \omega})^{*}$, $f=L_{v}+s$, we have $\|f\|=\|v\|_{\mathcal{M}_{\psi, \omega}}+\|s\|$, $\|f\|^{o}= \|v\|_{\mathcal{M}_{\psi, \omega}^{o}}+\|s\|^{o}$\\
\end{lemma}

\begin{lemma}\cite[Theorem 1.43, Theorem 1.44]{Chen1996}, \cite[Theroem 2.21]{Wang2024}
	For any $f\in \Lambda_{\varphi, \omega}$, $d(f)=d^{o}(f)=\theta(f)$, where
	$$
	d(f)=\inf\{\|f-f_{e}\|: \: f_{e}\in E_{\varphi,\omega}\};\;d^{o}(f)=\inf\{\|f-f_{e}\|_{\varphi, \omega}^{o}:\: f_{e}\in E_{\varphi, \omega}\}.
	$$
\end{lemma}

\begin{lemma}\cite[Chapter 2, Proposition 1.7]{Bennett1953}
\label{converge}
For $f\in L_{0}$ and $f_{n}\in L_{0}$,
we have
\begin{equation*}
|f|\leq \liminf_{n\rightarrow\infty} |f_{n}|\;\mu-a.e. \Rightarrow f^{*}\leq \liminf_{n\rightarrow} f_{n}^{*};
\end{equation*}
In particular,
$$
|f_{n}|\uparrow |f|\: \mu-a.e. \Rightarrow f^{*}\leq \liminf_{n\rightarrow\infty}f_{n}^{*}
$$
\end{lemma}

\section{Main Results}\label{MainResult}
\subsection{Extreme points in Orlicz-Lorentz spaces}\label{SecExtremepoint}

\begin{theorem}
Let $\varphi$ be an Orlicz function and $\omega$ be a decreasing weight. $x\in S(\Lambda_{\varphi, \omega})$ and $|x|=x^{*}(\sigma)$. Then $x$ is an extreme point of $B(\Lambda_{\varphi, \omega})$ if and only if\\
(1) $\rho_{\varphi, \omega}(x)=1$\\
(2) $x=\alpha \chi_{A}+s$ where $s(t)\in S$, $\alpha\in S^{\prime}$ and $\mu(\sigma(A)\cap L(\omega))=0$.\\
\end{theorem}
\begin{proof}
\textsl{Necessity}. The necessity of (1) is presented in \cite[Theorem 7]{kaminska1990extreme}. 
We will consider the necessary of $(2)$ on $R^{+}$. It has been proven in \cite{kaminska1990extreme} that $x(t)$ has at most one values not belonging to $S$. Therefore we only need to show the necessity of $\mu(\sigma(A)\cap L(\omega))=0$. If $\mu(\sigma(A)\cap L(\omega))\neq 0$, let $A_{0}$ satisfying $\sigma(A_{0})=\sigma(A)\cap L(\omega)$ and $\mu A_{0}>0$. Divide $A_{0}$ into $A_{1}$ and $A_{2}$ such that $\mu A_{1}=\mu A_{2}$. For arbitrary $\varepsilon$ such that $(\alpha-\varepsilon, \alpha+\varepsilon)\subset S^{\prime}$(In \cite[Theorem 7]{kaminska1990extreme}, both Case \uppercase\expandafter{\romannumeral1} and Case \uppercase\expandafter{\romannumeral2} in the proof of the necessity can ensure the existence of such $\varepsilon$). Define
\begin{align*}
x_{1}&=s+ (\alpha+\varepsilon) \chi_{A_{1}}+(\alpha-\varepsilon)\chi_{A_{2}}.\\
x_{2}&=s+ (\alpha-\varepsilon) \chi_{A_{1}}+(\alpha+\varepsilon)\chi_{A_{2}}.
\end{align*}
Obviously $2x=x_{1}+x_{2}$.
Let $\lambda_{1}<1<\lambda_{2}$ such that $\lambda_{2}(\alpha+\varepsilon)\in S^{\prime}$ and $\lambda_{1}(\alpha-\varepsilon)\in S^{\prime}$, for $\lambda\in (\lambda_{1}, \lambda_{2})$,
Then $\rho_{\varphi, \omega}(\lambda x_{1})=\rho_{\varphi, \omega}(\lambda x_{2})=\rho_{\varphi, \omega}(\lambda x)$ and $\|x_{1}\|_{\varphi, \omega}=\|x_{2}\|_{\varphi, \omega}=\|x\|_{\varphi, \omega}$, a contradiction.\\
\textsl{Sufficiency}.
The case that $\mu (L(\omega))=0$ is considered in \cite{kaminska1990extreme}.
Assume $\mu (L(\omega))\neq 0$ and $\mu(\sigma(A)\cap L(\omega))= 0$. For $x=s+\alpha \chi_{A}$, we can assume 
$x=\frac{y+z}{2}$ and $\rho_{\varphi, \omega}(y)=\rho_{\varphi, \omega}(z)=1$. We have
\begin{align*}
1
&=   \int_{0}^{\infty}\varphi(x(\sigma^{-1}(t)))\omega(t)dt\\
&=   \int_{0}^{\infty}\varphi(\frac{y(\sigma^{-1}(t))+z(\sigma^{-1}(t))}{2})\omega(t)dt\\
&\leq\frac{1}{2}\int_{0}^{\infty}(\varphi(y(\sigma^{-1}(t)))+\varphi(z(\sigma^{-1}(t))))\omega(t)dt\\
&\leq\frac{1}{2}\int_{0}^{\infty}(\varphi(y^{*}(t))+\varphi(z^{*}(t)))\omega(t)dt\\
&=\frac{1}{2}(\rho_{\varphi, \omega}(y)+\rho_{\varphi, \omega}(z))\\
&=1.
\end{align*}
Thus
\begin{align*}
\int_{R^{+}}\varphi(y(\sigma^{-1}(t)))\omega(t)dt&=\rho_{\varphi, \omega}(y)=1,\\ \int_{R^{+}}\varphi(z(\sigma^{-1}(t)))\omega(t)dt&=\rho_{\varphi, \omega}(z)=1.
\end{align*}
Since
\begin{align*}
0=\|y\|_{\varphi, \omega}+\|z\|_{\varphi, \omega}-2\|x\|_{\varphi, \omega}
&=    \int_{R^{+}}\left(\varphi(y(\sigma^{-1}(t)))+\varphi(z(\sigma^{-1}(t)))\right.\\
&\qquad\quad-\left.2\varphi(\frac{y(\sigma^{-1}(t))+z(\sigma^{-1}(t))}{2})\right)\omega(t)dt\\
&\geq\int_{\sigma(supp~s)}\left(\varphi(y(\sigma^{-1}(t)))+\varphi(z(\sigma^{-1}(t)))\right.\\
&\qquad\quad\left.-2\varphi(\frac{y(\sigma^{-1}(t))+z(\sigma^{-1}(t))}{2})\right)\omega(t)dt\\
&\geq 0.
\end{align*}
Thus
\begin{equation}
x(\sigma^{-1}(t))=y(\sigma^{-1}(t))=z(\sigma^{-1}(t))~on ~\sigma(supp~s).
\end{equation}
and
\begin{align*}
  \int_{R^{+}\backslash\sigma(supp~s)}\varphi(x(\sigma^{-1}(t)))\omega(t)dt
&=\int_{R^{+}\backslash\sigma(supp~s)}\varphi(y(\sigma^{-1}(t)))\omega(t)dt\\
&=\int_{R^{+}\backslash\sigma(supp~s)}\varphi(z(\sigma^{-1}(t)))\omega(t)dt.
\end{align*}
Then we will prove $x(t)=y(t)=z(t)$ on A.
First, we will prove $y(\sigma^{-1}(t))$ is non-increasing on $\sigma(A)$.
If not, since $\mu(\sigma(A)\cap L(\omega))=0$, then $\omega$ is strictly decreasing on $\sigma(A)$, there exists a measure preserving transformation $\sigma_{1}:R^{+}\rightarrow R^{+}$ such that $y(\sigma^{-1}(\sigma_{1}^{-1}(t)))$ is decreasing on $\sigma(A)$ and $y(\sigma^{-1}(\sigma_{1}^{-1}(t)))=y(\sigma^{-1}(t))$ on $R^{+}\backslash \sigma(supp~s)$. Thus
\begin{equation*}
\int_{R^{+}}\varphi(\sigma^{-1}(\sigma_{1}^{-1}(t)))\omega(t)dt>\int_{R^{+}}\varphi(y(\sigma^{-1}(t)))\omega(t)dt=\int_{R^{+}}\varphi(y^{*}(t))\omega(t)dt.
\end{equation*}
a contradiction. Thus $y(\sigma^{-1}(t))$ is decreasing on $\sigma(A)$. Using the same way, we can prove $z(\sigma^{-1}(t))$ is decreasing on $\sigma(A)$.
Besides,
\begin{equation*}
\alpha =x(\sigma^{-1}(t))=\frac{1}{2}y(\sigma^{-1}(t))+\frac{1}{2}z(\sigma^{-1}(t)).
\end{equation*}
Then we can deduce $y(\sigma^{-1}(t))$ and $z(\sigma^{-1}(t))$ are constant on $\sigma(A)$. If $y(\sigma^{-1}(t))\neq z(\sigma^{-1}(t))$ on $\sigma(A)$, assume $y(\sigma^{-1}(t))> z(\sigma^{-1}(t))$ without loss of generality. Then it is easy to obtain
\begin{equation*}
\int_{R^{+}}\varphi(y(\sigma^{-1}(t)))\omega(t)dt>\int_{R^{+}}\varphi(x(\sigma^{-1}(t)))\omega(t)dt=1.
\end{equation*}
A contradiction. Thus
\begin{equation*}
x(\sigma^{-1}(t))=y(\sigma^{-1}(t))=z(\sigma^{-1}(t))~on~\sigma(A).
\end{equation*}
It also implies $y(\sigma^{-1}(t))=z(\sigma^{-1}(t))=x(\sigma^{-1}(t))=0$ on $R^{+}\backslash \sigma(supp~s)\backslash\sigma(A)$.
We finally have
\begin{equation*}
x(t)=y(t)=z(t)~on~R^{+}.
\end{equation*}
It implies that $x$ is an extreme point of $B(\Lambda_{\varphi, \omega})$.
\end{proof}

\begin{theorem}\cite[Theroem 1, Remark 2]{WangSEP2023}
Let $\varphi$ be an Orlicz function and $\omega$ be a decreasing weight. Let $x\in\Lambda_{\varphi, \omega}^{o}$ and $|x|=x^{*}(\sigma)$.\\
(1)If $\lim_{t\rightarrow\infty}\frac{\varphi(t)}{t}=\infty$ or $\lim_{u\rightarrow\infty}\frac{\varphi(t)}{t}=B<\infty$ and $\psi(B)\int_{0}^{\mu(supp~x)}\omega(t)dt>1$, then $x$ is an extreme point of $B(\Lambda_{\varphi, \omega}^{o})$ if and only if $kx(t)=s(t)\in S$ for $t\in R^{+}$ or $kx(t)=\alpha\chi_{A}$ where $\alpha\in S^{\prime}$ and $\mu(\sigma(A)\cap L(\omega))=0$.\\
(2)If $\lim_{t\rightarrow\infty}\frac{\varphi(t)}{t}=B<\infty$ and $\psi(B)\int_{0}^{\mu(supp~x)}\omega(t)dt\leq 1$. Then $x$ is an extreme point of $B(\Lambda_{\varphi,\omega}^{o})$ if and only if $x=\alpha\chi_{A}$ where $\mu(\sigma(A)\cap L(\omega))=0$.
\end{theorem}
From the definition of $k^{*}$ and $k^{**}$, if $x\in B(\Lambda_{\varphi, \omega}^{o})$ is an extreme point of $S(\Lambda^{o}_{\varphi, \omega})$, then $K(x)$ is a singleton.
\subsection{Strongly extreme point}
\label{SectionStronglyextremepoint}
\begin{theorem}\label{3}
$x\in S(\Lambda_{\varphi, \omega})$ is a strongly extreme point of $B(\Lambda_{\varphi, \omega})$ if and only if
$\varphi\in\Delta_{2}$ and $x$ is an extreme point of $B(\Lambda_{\varphi, \omega})$.
\end{theorem}
\begin{proof}

\textsl{Sufficiency.~}
For arbitrary $\delta>0$,
there exists $M>0$ such that
\begin{equation}\label{14}
\mu\{t\in R^{+}: |y_{n}(t)|>M\}<\frac{\delta}{4},\;\mu\{t\in R^{+}: |z_{n}(t)|>M\}<\frac{\delta}{4}.
\end{equation}
For arbitrary $\varepsilon>0$, let
\begin{equation*}
B=\{(u, v): \: |u|\leq M, |v|\leq M, |u-v|\geq \varepsilon \}.
\end{equation*}
Since $B$ is a closed set and the continuous function
\begin{equation*}
\frac{2\varphi(\frac{1}{2}u+\frac{1}{2}v )}{\varphi(u)+\varphi(v)}<1,
\end{equation*}
then there exists $\delta_{1}>0$ such that
\begin{equation*}
\frac{2\varphi(\frac{1}{2}u+\frac{1}{2}v)}{\varphi(u)+\varphi(v)}\leq 1-\delta_{1}, \;(u, v)\in B.
\end{equation*}
Let
$$
B_{n}=\{t\in S:  |y_{n}(t)|\leq M, |z_{n}(t)|\leq M, |y_{n}(t)-z_{n}(t)|\geq \varepsilon\}.
$$
Then we will prove $\mu B_{n}\rightarrow \infty$ as $n\rightarrow\infty$. If not,
assume there exists $\delta_{2}$ such that $\mu B_{n}>\delta_{2}$. Since
\begin{equation}\label{7}
\begin{aligned}
\|y_{n}\|_{\varphi, \omega}+\|z_{n}\|_{\varphi, \omega}-2\|x\|_{\varphi, \omega}
&=\int_{R^{+}}\varphi(y_{n}^{*}(t))\omega(t)dt+
\int_{R^{+}}\varphi(z_{n}^{*}(t))\omega(t)dt\\
&\qquad-2\int_{R^{+}}\varphi(x^{*}(t))\omega(t)dt\\
&\geq\int_{R^{+}}\varphi(y_{n}(\sigma^{-1}(t)))\omega(t)dt
+\int_{R^{+}}\varphi(z_{n}(\sigma^{-1}(t)))\omega(t)dt\\
&\qquad-2\int_{R^{+}}\varphi\left( \left(\frac{y_{n}+z_{n}}{2}\right)\sigma^{-1}(t)\right)\omega(t)dt\\
&=\int_{R^{+}}\varphi(y_{n}(t))\omega(\sigma(t))dt+
\int_{R^{+}}\varphi(z_{n}(t))\omega(\sigma(t))dt\\
&\qquad-2\int_{R^{+}}\varphi(\frac{y_{n}+z_{n}}{2})\omega(\sigma(t))dt\\
&\geq \int_{B_{n}}\left(
\varphi(y_{n})+\varphi(z_{n})-2\varphi(\frac{y_{n}+z_{n}}{2})\right)
\omega(\sigma(t))dt\\
&\geq 0.
\end{aligned}
\end{equation}
we obtain
\begin{align}\label{89}
\lim_{n\rightarrow\infty}\int_{R^{+}}\varphi(y_{n}(\sigma^{-1}(t)))\omega(t)dt
&=\lim_{n\rightarrow\infty} \rho_{\varphi, \omega}(y_{n})=1.\\
\label{9}\lim_{n\rightarrow\infty}\int_{R^{+}}\varphi(z_{n}(\sigma^{-1}(t)))\omega(t)dt
&=\lim_{n\rightarrow\infty} \rho_{\varphi, \omega}(z_{n})=1.
\end{align}
Besides, from (\ref{7}) we can find $n_{0}\in N$ and $M_{n_{0}}>M>0$ such that
\begin{align}\label{1933}
\int_{B_{n_{0}}}(\varphi(y_{n_{0}})+\varphi(z_{n_{0}})
-2\varphi(\frac{y_{n_{0}}+z_{n_{0}}}{2}))\omega(\sigma(t))dt
\leq\frac{\varepsilon(W(M_{n_{0}}+\delta_{2})-W(M_{n_{0}}))}{2}
\end{align}
and from formula (\ref{89}) and (\ref{9})
\begin{align*}
y_{n_{0}}(\sigma^{-1}(t))<\frac{\varepsilon}{2}, \;t\geq M_{n_{0}},\\
z_{n_{0}}(\sigma^{-1}(t))<\frac{\varepsilon}{2}, \;t\geq M_{n_{0}}.
\end{align*}
Notice that $\sigma(B_{n_{0}})\subset (0, M_{n_{0}})$. If not, there exists $t_{0}\in B_{n_{0}}$ such that $\sigma(t_{0})> M_{n_{0}}$. Then
\begin{align*}
y_{n_{0}}(\sigma^{-1}(\sigma(t_{0})))=y_{n_{0}}(t_{0})
&<\frac{\varepsilon}{2}, \\
z_{n_{0}}(\sigma^{-1}(\sigma(t_{0})))=z_{n_{0}}(t_{0})
&<\frac{\varepsilon}{2}.\\
\end{align*}
and therefore
\begin{equation*}
|y_{n_{0}}(t_{0})-z_{n_{0}}(t_{0})|<\frac{\varepsilon}{2}+\frac{\varepsilon}{2}
=\varepsilon, t_{0}\in B_{n_{0}},
\end{equation*}
a contradiction. Thus for $t\in B_{n_{0}}$
\begin{align*}
\int_{B_{n_{0}}}\left(
\varphi(y_{n_{0}})+\varphi(z_{n_{0}})
-2\varphi\left(\frac{y_{n_{0}}+z_{n_{0}}}{2}\right)
\right)\omega(\sigma(t))dt
&\geq\varepsilon(W(M_{n_{0}}+\delta_{2})-W(M_{n_{0}}))
\end{align*}
A contradiction to formula (\ref{1933}).
Then
$\mu B_{n}\rightarrow\infty$. For the $\delta$ mentioned above, we assume $\mu B_{n}<\frac{\delta}{4}$ without loss of generality. Then from formula (\ref{14}) we obtain
\begin{equation*}
\mu\{t\in supp~s: |y_{n}(t)-z_{n}(t)|\geq\varepsilon \}<\delta.
\end{equation*}
It implies
\begin{align*}
y_{n}                &\stackrel{\mu}\rightarrow x,    ~z_{n}                \stackrel{\mu}\rightarrow x,     on        ~supp~s.\\
y_{n}(\sigma^{-1}(t))&\stackrel{\mu}\rightarrow x^{*},~z_{n}(\sigma^{-1}(t))\stackrel{\mu}\rightarrow~x^{*}, on ~\sigma(supp~s).
\end{align*}
Let $u_{n}=\inf_{m\geq n}y_{m}\chi_{supp~s}$, $v_{n}=\sup_{m\leq n} y_{m}\chi_{supp~s}$, then $u_{n}\uparrow x$ and $v_{n}\downarrow x$ on $supp~s$.
Then from Lemma \ref{converge}
 $$
 u_{n}^{*}\uparrow x^{*},~ v_{n}^{*}\downarrow x^{*} ~on~ \sigma(supp~s)~(n\rightarrow \infty).
 $$
 From Fatou Lemma
\begin{align*}
\int_{\sigma(supp~s)}\varphi(y_{n}^{*})\omega(t)dt
&\geq \lim_{n\rightarrow\infty}\int_{\sigma(supp~s)}\varphi(u_{n}^{*})\omega(t)dt
\rightarrow\int_{\sigma(supp~x)}\varphi(x^{*}(t))\omega(t)dt;\\
\int_{\sigma(supp~s)}\varphi(y_{n}^{*})\omega(t)dt
&\leq \lim_{n\rightarrow\infty}\int_{\sigma(supp~s)}\varphi(v_{n}^{*})\omega(t)dt
\rightarrow\int_{\sigma(supp~s)}\varphi(x^{*}(t))\omega(t)dt.
\end{align*}
as $n\rightarrow\infty$. Thus
\begin{align*}
\lim_{n\rightarrow\infty}\int_{\sigma(supp~s)}\varphi(y_{n}^{*})\omega(t)dt
&=\int_{\sigma(supp~s)}\varphi(x^{*})\omega(t)dt,\\
\lim_{n\rightarrow\infty}\int_{\sigma(supp~s)}\varphi(z_{n}^{*})\omega(t)dt
&=\int_{\sigma(supp~s)}\varphi(x^{*})\omega(t)dt.
\end{align*}
Then we will prove
\begin{equation*}
y_{n}^{*}(t)+z_{n}^{*}(t)\stackrel{\mu}\rightarrow 2x^{*}(t)~on~R^{+}\backslash \sigma(supp~s).
\end{equation*}
If not there exists $\varepsilon_{0}>0$ and $\varepsilon_{1}>0$ such that $\mu\{t\in R^{+}: |y_{n}^{*}(t)+z_{n}^{*}(t)-2x^{*}(t)|\geq \varepsilon_{0}\}>\varepsilon_{1}$.
If $\lim_{n\rightarrow\infty}\mu\{t\in R^{+}: y_{n}^{*}(t)+z_{n}^{*}(t)-2x^{*}(t)\geq \varepsilon_{0}\}\neq 0$. Then there exists $t_{0}\in \sigma(A)$ such that
$$
y_{n}^{*}(t_{0})+z_{n}^{*}(t_{0})> 2x^{*}(t_{0})+\varepsilon_{0}.
$$
Thus there exists $\beta>0$ such that
\begin{align*}
      \int_{R^{+}\backslash \sigma(supp~s)}(\varphi(y_{n}^{*}(t))+\varphi(z_{n}^{*}(t)))\omega(t)dt
&\geq \int_{                \sigma(A)     }(\varphi(y_{n}^{*}(t))+\varphi(z_{n}^{*}(t)))\omega(t)dt\\
&\geq    \int_{\{t\in \sigma(A): t\leq t_{0}\}}(\varphi(y_{n}^{*}(t_{0}))+\varphi(z_{n}^{*}(t_{0})))\omega(t)dt\\
&
\qquad +2\int_{\{t\in \sigma(A): t\geq t_{0}\}}(\varphi(\alpha))\omega(t)dt+\beta\\
&>2\int_{\sigma(A)}\varphi(x_{0})\omega(t)dt+\beta.
\end{align*}
This lead to a contradiction to the fact
\begin{align*}
 \lim_{n\rightarrow\infty}\int_{R^{+}\backslash \sigma(supp~s)}\left( \varphi(y_{n}^{*})+\varphi(z_{n}^{*}) \right)\omega(t)dt
 &=\lim_{n\rightarrow\infty}2-\int_{\sigma(supp~s)}\left( \varphi(y_{n}^{*})+\varphi(z_{n}^{*}) \right)\\
 &= 2-2\int_{\sigma(supp~s)}\varphi(x^{*})\omega(t)dt\\
 &=2\int_{\sigma(A)}\varphi(x^{*})\omega(t)dt.
\end{align*}
Thus, for arbitrary $\varepsilon>0$, $\lim_{n\rightarrow\infty}\mu\{t\in \sigma(A): y_{n}^{*}(t)+z_{n}^{*}(t)-2x^{*}(t)\geq \varepsilon \}=0$.
If $\mu\{t\in \sigma(A): y_{n}^{*}(t)+z_{n}^{*}(t)-2x^{*}(t)\geq \varepsilon_{0} \}\not\rightarrow 0$,
Thus $2\int_{\sigma(A)}\varphi(x^{*}(t))\omega(t)dt<\int_{\sigma(A)}\varphi(y_{n}^{*}(t))\omega(t)dt+\int_{\sigma(A)}\varphi(z_{n}^{*}(t))\omega(t)dt$, and it implies $\rho_{\varphi, \omega}(y_{n})+\rho_{\varphi, \omega}(z_{n})\not\rightarrow 2$, a contradiction.
Then $y_{n}(t)+z_{n}(t)\stackrel{\mu}\rightarrow 2x^{*}(t)$ on $\sigma(A)$.
Therefore
\begin{equation*}
y_{n}^{*}(t)+z_{n}^{*}(t)\stackrel{\mu}\rightarrow 2x^{*}(t)=2\alpha\; on\; \sigma(A).
\end{equation*}
Given the fact that $y_{n}^{*}$ and $z_{n}^{*}$ are decreasing on $\sigma(A)$ we have
$$
y_{n}^{*}(t)\rightarrow\alpha, ~z_{n}^{*}(t)\rightarrow \alpha~a.e. ~on ~\sigma(A).
$$
Since
\begin{align*}
 \lim_{n\rightarrow\infty}\int_{\sigma(A)}\left(\varphi(y_{n}(\sigma^{-1}(t)))+\varphi(z_{n}(\sigma^{-1}(t)))\right)\omega(t)dt
&\geq2\liminf_{n\rightarrow\infty}\varphi(x^{*}(t))\omega(t)dt\\
&=   2\alpha\mu A.
\end{align*}
Thus
\begin{align}
\label{s111}
&\lim_{n\rightarrow\infty}\int_{\sigma(A)}\varphi(y_{n}(\sigma^{-1}(t)))\omega(t)dt=
\lim_{n\rightarrow\infty}\int_{\sigma(A)}\varphi(z_{n}(\sigma^{-1}(t)))\omega(t)dt=
\alpha\mu A.\\
\label{a111}
&\lim_{n\rightarrow\infty}\int_{R^{+}\backslash\sigma(supp~s)\backslash\sigma(A)}\varphi(y_{n}(\sigma^{-1}(t)))\omega(t)dt=
\lim_{n\rightarrow\infty}\int_{R^{+}\backslash\sigma(supp~s)\backslash\sigma(A)}\varphi(z_{n}(\sigma^{-1}(t)))\omega(t)dt=
0
\end{align}
Besides, from the fact that
\begin{align}
\label{0111}
\int_{R^{+}\backslash \sigma(supp~s)\backslash \sigma(A)}\varphi(y_{n}^{*}(t))\omega(t)dt\rightarrow 0, \\
\label{0222}
\int_{R^{+}\backslash \sigma(supp~s)\backslash \sigma(A)}\varphi(z_{n}^{*}(t))\omega(t)dt\rightarrow 0.
\end{align}
Therefore from formula (\ref{s111}, \ref{a111}, \ref{0111}, \ref{0222}) we have
\begin{align*}
\lim_{n\rightarrow\infty}\rho_{\varphi, \omega}(y_{n}-z_{n})
&=\lim_{n\rightarrow\infty}\int_{R^{+}}\varphi((y_{n}-z_{n})^{*}(t))\omega(t)dt\\
&=\lim_{n\rightarrow\infty}\int_{\sigma(supp~s)}\varphi((y_{n}-z_{n})^{*}(t))\omega(t)dt
+\int_{\sigma(A)}\varphi((y_{n}-z_{n})^{*}(t))\omega(t)dt\\
&\qquad+\int_{R^{+}\backslash\sigma(A)\backslash \sigma(supp~s)}\varphi((y_{n}-z_{n})^{*}(t))\omega(t)dt\\
&\leq\lim_{n\rightarrow\infty}\int_{\sigma(supp~s)}\varphi((y_{n}-z_{n})^{*}(t))\omega(t)dt
+\int_{\sigma(A)}\varphi((y_{n}-z_{n})^{*}(t))\omega(t)dt\\
&\qquad+\int_{R^{+}\backslash\sigma(A)\backslash \sigma(supp~s)}\varphi((y_{n})^{*}(t))\omega(t)dt
+\int_{R^{+}\backslash\sigma(A)\backslash \sigma(supp~s)}\varphi((z_{n})^{*}(t))\omega(t)dt\\
&=0.
\end{align*}
Given $\varphi\in\Delta_{2}$, we obtain $\|y_{n}-z_{n}\|_{\varphi, \omega}\rightarrow 0$. Besides, it follows from the proof process that
\begin{equation*}
y_{n}(\sigma^{-1}(t))\stackrel{\mu}\rightarrow y_{n}^{*}(t),\; z_{n}(\sigma^{-1}(t))\stackrel{\mu}\rightarrow z_{n}^{*}(t).
\end{equation*}
\textsl{Necessity.}~
It suffices to prove the necessity of $\varphi\in\Delta_{2}$. If $\varphi\notin\Delta_{2}$, there exists  $\{u_{n}\}_{n\in N}$, $u_{n}\uparrow\infty$ as $n\rightarrow\infty$ and $\varphi((1+\frac{1}{n})u_{n})> 2^{n}\varphi(u_{n})$. Let $G_{n}=[0, d_{n}]$ where $W(d_{n})=\frac{1}{2^{n}\varphi(u_{n})}$, $(n=1,2...)$. Obviously, $\{d_{n}\}_{n\in N}$ is bounded. For $x_{0}\in S(\Lambda_{\varphi, \omega})$ there exists $\gamma>0$ such that $0<|f_{0}(t)|<\gamma$ for $t\in G_{n}$. Let
\begin{align*}
y_{n}(t)=(x_{0}(t)+u_{n})\chi_{G_{n}}+x_{0}(t)\chi_{R^{+}\backslash G_{n}},\\
z_{n}(t)=(x_{0}(t)-u_{n})\chi_{G_{n}}+x_{0}(t)\chi_{R^{+}\backslash G_{n}}.
\end{align*}
Obviously, $y_{n}+z_{n}=2x_{0}$. Since
\begin{align*}
\rho_{\varphi, \omega}(y_{n})\leq \rho_{\varphi, \omega}(x)+\frac{1}{2^{n}},\\
\rho_{\varphi, \omega}(z_{n})\geq \rho_{\varphi, \omega}(x)-\frac{1}{2^{n}}.
\end{align*}
Thus
\begin{align*}
\rho_{\varphi, \omega}((1+\frac{1}{2^{n}})y_{n})&\leq 1,\\
\rho_{\varphi, \omega}((1-\frac{1}{2^{n}})z_{n})&\geq 1.
\end{align*}
Thus
\begin{equation*}
\|y_{n}\|_{\varphi, \omega}\rightarrow 1, \|z_{n}\|_{\varphi, \omega}\rightarrow 1~as~n\rightarrow \infty.
\end{equation*}
Since $\rho_{\varphi, \omega}(y_{n}-z_{n})=2\int_{G_{n}}\varphi(u_{n})\omega(t)dt\leq 1 $ and
\begin{equation*}
\rho_{\varphi, \omega}((1+\frac{1}{n})(y_{n}-z_{n}))\geq 2^{n}\int_{G_{n}}\varphi(u_{n})\omega(t)dt\geq 1.
\end{equation*}
Thus $\|y_{n}-z_{n}\|_{\varphi, \omega}\not\rightarrow 0(n\rightarrow\infty)$.

\end{proof}

\begin{theorem}\cite{WangSEP2023}
Assume $\varphi$ be an Arbitrary Orlicz function and $\omega$ be a decreasing weight. Then $x\in S(\Lambda_{\varphi, \omega}^{o})$ is an strongly extreme point of $B(\Lambda_{\varphi, \omega}^{o})$ if and only if $\varphi\in\Delta_{2}$ and $x$ is an extreme point.
\end{theorem}
\begin{proof}
\textsl{Sufficiency. }
We employ a method different from that in \cite{WangSEP2023} to prove the sufficiency.
At the beginning we will consider the case $\lim_{t\rightarrow\infty}\frac{\varphi(t)}{t}=\infty$ or $\lim_{u\rightarrow\infty}\frac{\varphi(t)}{t}=B<\infty$ and $\psi(B)\int_{0}^{\mu(supp~x)}\omega(t)dt>1$, therefore
 $K(x)\neq \emptyset$.
Assume $x=\frac{y_{n}+z_{n}}{2}$ and $\|y_{n}\|_{\varphi, \omega}^{o}=\frac{1}{k_{n}}\left( 1+\rho_{\varphi, \omega}(k_{n}y_{n}) \right)$, $\|z_{n}\|_{\varphi, \omega}^{o}=\frac{1}{h_{n}}\left( 1+\rho_{\varphi, \omega}(h_{n}z_{n})\right)$.
Since $x$ is an extreme point, we will first discuss the case that $kx(t)\in S$ where $k\in K(x)$. Following the method of \cite{WangSEP2023}, $\{k_{n}\}_{n\in N}$ and $\{h_{n}\}_{n\in N}$ are bounded. Then for arbitrary $\delta>0$ we have
\begin{equation}
\mu\{t\in R^{+}: |y_{n}(t)|>M\}<\frac{\delta}{4},\;\mu\{t\in R^{+}: |z_{n}(t)|>M\}<\frac{\delta}{4}.
\end{equation}
Since
\begin{align*}
\|y_{n}\|_{\varphi, \omega}^{o}+\|z_{n}\|_{\varphi, \omega}^{o}-2\|x\|_{\varphi,\omega}^{o}
&\geq\frac{1}{k_{n}}(1+\rho_{\varphi, \omega}(k_{n}y_{n}))+\frac{1}{h_{n}}(1+\rho_{\varphi, \omega}(h_{n}z_{n}))\\
&\qquad-2\frac{k_{n}+h_{n}}{2k_{n}h_{n}}(1+\rho_{\varphi, \omega}(\frac{2k_{n}h_{n}}{k_{n}+h_{n}}x))\\
&=\frac{k_{n}+h_{n}}{k_{n}h_{n}}
\left(\frac{h_{n}}{k_{n}+h_{n}}
\int_{R^{+}}\varphi(k_{n}y_{n}^{*}(t))\omega(t)dt\right.\\
 &\qquad+\frac{k_{n}}{k_{n}+h_{n}}
 \int_{R^{+}}\varphi(h_{n}z_{n}^{*}(t))\omega(t)dt \\ &\left.\qquad-\varphi(\frac{2k_{n}h_{n}}{k_{n}+h_{n}}x(\sigma^{-1}(t)))\omega(t)dt\right)\\
 &\geq\frac{k_{n}+h_{n}}{k_{n}h_{n}}\left(\frac{h_{n}}{k_{n}+h_{n}}\int_{R^{+}}\varphi(k_{n}y_{n}(\sigma^{-1}(t)))\omega(t)dt\right.\\
 &\qquad+\frac{k_{n}}{k_{n}+h_{n}}\int_{R^{+}}\varphi(h_{n}z_{n}(\sigma^{-1}(t)))\omega(t)dt\\
 &\left.\qquad-\varphi(\frac{k_{n}h_{n}}{k_{n}+h_{n}}y_{n}(\sigma^{-1}(t))
 +z_{n}(\sigma^{-1}(t)))\omega(t)dt\right)\\
 &\geq\frac{k_{n}+h_{n}}{k_{n}h_{n}}\int_{B_{n}}\left( \frac{h_{n}}{k_{n}+h_{n}}\varphi(k_{n}y_{n})
 +\frac{k_{n}}{k_{n}+h_{n}}\varphi(h_{n}z_{n})\right.\\
 &\qquad\left.-\varphi(\frac{h_{n}}{k_{n}+h_{n}}k_{n}y_{n}
 +\frac{k_{n}}{k_{n}+h_{n}}h_{n}z_{n})\right)\omega(\sigma(t))dt\\
 &\geq 0.
\end{align*}
Then following the method of Theorem \ref{3}, we can obtain the conclusion that
\begin{equation*}
k_{n}y_{n}\stackrel{\mu}\rightarrow kx;\;
h_{n}z_{n}\stackrel{\mu}\rightarrow kx\;\;(n\rightarrow\infty).
\end{equation*}
Then
\begin{align*}
\lim_{n\rightarrow\infty}k_{n}=\lim_{n\rightarrow\infty}\|k_{n}y_{n}\|=k\|x\|=k.\\
\lim_{n\rightarrow\infty}h_{n}=\lim_{n\rightarrow\infty}\|h_{n}z_{n}\|=k\|x\|=k.
\end{align*}
Thus
\begin{equation*}
\|y_{n}-z_{n}\|=\lim_{n\rightarrow\infty}\|k_{n}y_{n}-h_{n}z_{n}\|=0
\end{equation*}
and $x$ is an strongly extreme point of $B(\Lambda_{\varphi, \omega}^{o})$.

If $x\in S(\Lambda_{\varphi, \omega}^{o})$ and $kx=\alpha\chi_{A}$ where $\alpha\in S^{\prime}$, we have
\begin{equation}\label{11}
\begin{aligned}
2
&\geq \|y_{n}\|_{\varphi, \omega}^{o}+\|z_{n}\|_{\varphi, \omega}^{o}\\
&\geq\frac{k_{n}+h_{n}}{k_{n}h_{n}}\left( 1+\int_{R^{+}}\left(
\frac{h_{n}}{k_{n}+h_{n}}\varphi(k_{n}y_{n}^{*}(t))+\frac{k_{n}}{k_{n}+h_{n}}\varphi(h_{n}z_{n}^{*}(t))
\right)\omega(t)dt \right)\\
&\geq\frac{k_{n}+h_{n}}{k_{n}h_{n}}
\left(1+\int_{R^{+}}\left(\frac{h_{n}}{k_{n}+h_{n}}\varphi(k_{n}y_{n}(\sigma^{-1}(t)))
+\frac{k_{n}}{k_{n}+h_{n}}\varphi(h_{n}z_{n}(\sigma^{-1}(t)))\right)\omega(t)dt \right)\\
&\geq\frac{k_{n}+h_{n}}{k_{n}h_{n}}
\left(\varphi
\left(
\frac{k_{n}h_{n}}{k_{n}+h_{n}}(y_{n}+z_{n})
\right)
\right)\omega(t)dt\\
&\geq 2\frac{k_{n}+h_{n}}{2k_{n}h_{n}}\int_{R^{+}}\varphi\left(
\frac{2k_{n}h_{n}}{k_{n}+h_{n}}x^{*}(t)\right)\omega(t)dt\\
&\geq 2\|x\|_{\varphi, \omega}^{o}\\
&=2.
\end{aligned}
\end{equation}
Then $\frac{2k_{n}h_{n}}{k_{n}+h_{n}}\in K(x)$ and $\frac{2k_{n}h_{n}}{k_{n}+h_{n}}x^{*}(t)=\alpha\chi_{A}$. Then
\begin{align*}
\frac{h_{n}}{k_{n}+h_{n}}\varphi(k_{n}y^{*}_{n}(t))
+\frac{k_{n}}{k_{n}+h_{n}}\varphi(h_{n}z^{*}_{n}(t))=\varphi(\alpha), \;\;t\in [0, \mu A],\\
\frac{h_{n}}{k_{n}+h_{n}}\varphi(k_{n}y_{n}^{*}(t))+\frac{k_{n}}{k_{n}+h_{n}}\varphi(h_{n}z_{n}^{*}(t))=\varphi(0), \;\ t\notin [0, \mu A].
\end{align*}
For an $n\in N$, since $k_{n}y_{n}^{*}(t)$ and $h_{n}g_{n}^{*}(t)$ are decreasing,
it is obviously that there exist constants $\alpha_{n}^{1}$ and $\alpha_{n}^{2}$ such that $k_{n}y_{n}^{*}(t)=\alpha_{1}\chi_{[0, \mu A]}$ and $h_{n}z_{n}^{*}(t)=\alpha_{2}\chi_{[0, \mu A]}$.
Then we have
\begin{equation*}
\|y_{n}\|_{\varphi, \omega}^{o}=\alpha_{n}^{1}\|\chi_{A}\|_{\varphi, \omega}^{o}\rightarrow 1\:(n\rightarrow \infty),\\
\|z_{n}\|_{\varphi, \omega}^{o}=\alpha_{n}^{2}\|\chi_{A}\|_{\varphi, \omega}^{o}\rightarrow 1\:(n\rightarrow \infty).
\end{equation*}
Thus $\alpha_{n}^{1}-\alpha_{n}^{2}\rightarrow 0$. From formula (\ref{11}), we have $y_{n}^{*}(t)-y_{n}(\sigma^{-1}(t))\rightarrow 0$ as $n\rightarrow\infty$.
\begin{align*}
\rho_{\varphi, \omega}(y_{n}-z_{n})
&=\rho_{\varphi, \omega}(y_{n}(\sigma^{-1}(t))-z_{n}(\sigma^{-1}(t)))\rightarrow 0\;\: (n\rightarrow\infty).
\end{align*}
Thus $x$ is a strongly exposed point of $B(\Lambda_{\varphi, \omega})$.

Finally we consider the case that $\lim_{t\rightarrow\infty}\frac{\varphi(t)}{t}=B<\infty$ and $\psi(B)\int_{0}^{\mu(supp~x)}\omega(t)dt\leq 1$. The proof process is obviously.\\
\textsl{Necessity.} Please refer to \cite{WangSEP2023}.
\end{proof}
\subsection{Norm expression of $\|\cdot\|_{\mathcal{M}_{\varphi, \omega}}$}\label{Normexpression}
\begin{lemma}\cite{2019abstractlorentz}
Let $\varphi$ be an Orlicz function and $f\in L_{0}$ such that $P_{\varphi, \omega}(f)<\infty$.
Then the inverse function $\omega^{f^{*}}$ satisfies $\omega^{f^{*}}\prec \omega$ and
\begin{align*}
&P_{\varphi, \omega}(f)=\int_{R^{+}}\varphi(\frac{(f^{*})^{0}}{\omega})\omega(t)dt
=\int_{R^{+}}\varphi(\frac{f^{*}}{\omega^{f^{*}}})\omega^{f^{*}}(t)dt,\\
&\rho_{\psi, \omega}\left(p\left(\frac{(f^{*})^{0}}{\omega}\right)\right)
=\int_{R^{+}}\psi\left(p\left(\frac{f^{*}}{\omega^{f^{*}}}\right)\right)
\omega^{f^{*}}(t)dt.
\end{align*}
\end{lemma}Let explore the expressions of Orlicz norm in $\mathcal{M}_{\varphi, \omega}^{o}$ generated by arbitrary Orlicz function.
\begin{theorem}\label{Norminkothedual}
Assume $\varphi$ is an Orlicz function and $\omega$ is a decreasing weight. Let $v\in\mathcal{M}_{\varphi, \omega}^{o}$. \\
(1)If Orlicz function $\varphi$ is an $N$-function, then $k^{**}_{\mathcal{M}}<\infty$\\
(2)If $\varphi$ satisfies $\lim_{u\rightarrow \infty}$ $\frac{\varphi(t)}{t}=B<\infty$ and $\lim_{t\rightarrow\infty}Bt-\varphi(t)=\infty$, then $k^{**}_{\mathcal{M}}(v)<\infty$.\\
(3)If $\varphi$ satisfies $\lim_{t\rightarrow\infty}\frac{\varphi(t)}{t}=B<\infty$ and $\lim_{t\rightarrow \infty}Bt-\varphi(t)<\infty$, $v\in\mathcal{M}_{\varphi, \omega}^{o}$ satisfies $\psi(B)\int_{0}^{\mu(supp~v)}\omega(t)dt>1$, then $k_{\mathcal{M}}^{**}<\infty$.\\
(4)If $\varphi$ satisfies $\lim_{t\rightarrow\infty}\frac{\varphi(t)}{t}=B<\infty$ and $\lim_{t\rightarrow\infty}Bt-\varphi(t)<\infty$. $v\in\mathcal{M}_{\varphi, \omega}^{o}$ and $\psi(B)\int_{0}^{\mu(supp~v)}\omega(t)dt\leq 1$, then $K_{\mathcal{M}}(v)=\emptyset$.\\
In case (1)(2)(3),
$\|v\|_{\mathcal{M}_{\varphi, \omega}^{o}}=\frac{1}{k}(1+P_{\varphi, \omega}(kx))$ where $v\in K_{\mathcal{M}}(v)$.
In case (4), $\|v\|_{\mathcal{M}_{\varphi, \omega}^{o}}=B\int_{0}^{\mu(supp~v)}v^{*}(t)dt$.
\end{theorem}
\begin{proof}
(1)The proof is in \cite[Theorem 2.11]{Wang2024}.\\
(2)
The condition $\lim_{t\rightarrow\infty}Bt-\varphi(t)=\infty$ is equivalent to $\psi(B)=\infty$. For all $v\in\mathcal{M}_{\varphi, \omega}^{o}$ and $v\neq 0$, from \cite[Theorem 3.6]{Halperin195305} there exists $k_{v}$ and $t_{v}$ such that
\begin{equation*}
\int_{0}^{t_{v}}\psi\left(p\left(\frac{k_{v}(v^{*})^{0}(t_{v})}{\omega(t_{v})}\right)\right)\omega(t_{v})dt\geq1
\end{equation*}
thus $k^{*}_{\mathcal{M}}\leq k^{**}_{\mathcal{M}}\leq k_{v}$.\\
(3)
Let $a\in R^{+}$ such that
\begin{equation*}
\psi(B)\int_{0}^{a}\omega(t)dt>1
\end{equation*}
Since $\lim_{u\rightarrow\infty}p(u)=B$, there exists $u_{v}>0$ for which $\psi(p(u_{v}))\int_{0}^{a}\omega(t)dt>1$. Define $k^{\prime}_{v}=\frac{u_{v}\omega(a)}{k(f^{*})^{o}(a)}$ then
\begin{equation*}
\int_{R^{+}}\psi(p(\frac{u_{v}\omega(a)}{(f^{*})^{0}(a)}
\frac{(f^{*})^{0}(t)}{\omega(t)}))\omega(t)dt\geq \int_{R^{+}}\psi(p(u_{v}))\omega(t)dt\geq 1
\end{equation*}
Thus $k^{*}_{\mathcal{M}}\leq k^{**}_{\mathcal{M}}\leq k_{v}$.\\
(4)
Note that $\psi(B)<\infty$ and $\psi(u)=\infty$ for any $u>B$. Hence, for any $x\in \mathcal{M}_{\varphi, \omega}$ such that $\rho_{\psi, \omega}(x)\leq 1$, we get that $x(t)\leq B$ a.e. in $R^{+}$ and
\begin{equation*}
\int_{R^{+}}x(t)v(t)dt\leq \int_{0}^{\infty}x^{*}(t)v^{*}(t)dt\leq B\int_{0}^{\infty}v^{*}(t)dt
\end{equation*}
For $x_{0}=B\chi_{[0, \mu(supp~v)]}$ we have $\rho_{\psi, \omega}(x_{0})=\psi(B)\int_{0}^{\mu(supp~v)}\omega(t)dt\leq 1$ and
\begin{equation*}
\int_{R^{+}}x(t)v(t)dt=\int_{0}^{\infty}x^{*}(t)v^{*}(t)dt=B\int_{0}^{\mu(supp~v)}\omega(t)dt
\end{equation*}
Thus $\|v\|_{\mathcal{M}_{\varphi, \omega}^{o}}=B\int_{0}^{\mu(supp~v)}v^{*}(t)dt$.
\end{proof}
\subsection{Supporting functionals}\label{Supportingfunctional}
\begin{theorem}
Let $\varphi$ be an Orlicz function satisfying
$\lim_{u\rightarrow\infty}\frac{\psi(u)}{u}=B<\infty$ and $\lim_{u\rightarrow\infty}Bu-\psi(u)<\infty$.
If $v\in\mathcal{M}_{\varphi, \omega}^{o}$ attain its norm at $x\in \Lambda_{\varphi,\omega}$, then $\|x\|_{\varphi, \omega}=\frac{1}{B}x^{*}(0)$.
\end{theorem}
\begin{proof}
From \cite[Corallary 6.16, Theorem 8.9]{2019abstractlorentz} we have
\begin{align*}
\|x\|_{\Lambda_{\varphi, \omega}}
&=\|x\|_{(Q_{L_{\varphi}, \omega})^{\prime}}\\
&=\sup\left\{\int_{R^{+}}v    (t)x    (t)dt:~\|v\|_{\mathcal{M}_{\varphi, \omega}^{o}}=1\right\}\\
&=\sup\left\{\int_{R^{+}}v^{*}(t)x^{*}(t)dt:~\|v\|_{\mathcal{M}_{\varphi, \omega}^{o}}=1\right\}\\
&=\sup\left\{\int_{R^{+}}v^{*}(t)x^{*}(t)dt: ~B\int_{R^{+}}v^{*}(t)dt=1 \right\}\\
&=B\|\frac{v^{*}}{\omega^{v^{*}}}\|_{\Lambda_{\omega^{v^{*}}}}
\cdot\frac{1}{B}\|x^{*}(t)\omega^{v^{*}}(t)\|_{M_{W^{v^{*}}}}\\
&=\frac{1}{B}\|x^{*}(t)\omega^{v^{*}}(t)\|_{M_{W^{v^{*}}}}\\
&=\frac{1}{B}\sup_{\alpha\in R^{+}}\frac{\int_{0}^{\alpha}x^{*}(t)\omega^{v^{*}}(t)dt}{\int_{0}^{\alpha}\omega^{v^{*}}(t)dt}\\
&=\frac{1}{B}x^{*}(0).
\end{align*}
\end{proof}
\begin{theorem}
Let $\omega$ be an arbitrary decreasing weight. Assume $\varphi$ satisfying  $\lim_{t\rightarrow\infty}\frac{\psi(t)}{t}=B<\infty$ and $\lim_{t\rightarrow\infty}Bt-\psi(t)<\infty$. 
For $v\in\mathcal{M}_{\psi, \omega}$ and $\psi(B)\int_{0}^{\mu(supp~v)}\omega(t)dt\leq 1$, then $v$ is norm attainable at $x\in S(\Lambda_{\varphi, \omega})$ if and only if \\
(1)$\rho_{\varphi, \omega}(x)=1$ and\\
(2)$x(t)=B\chi_{supp~v}(t)+x_{0}(t)$ where $(supp~x_{0})\cap (supp~v)=\emptyset$ and $|x_{0}(t)|\leq B$.
\end{theorem}
\begin{proof}
\textsl{Necessity:}
Since $\varphi\in\Delta_{2}$, then $x\in S(\Lambda_{\varphi, \omega})$ is equivalent to $\rho_{\varphi, \omega}(x)=1$.\\
Assume $v$ attain its norm at $x\in S(\Lambda_{\varphi, \omega})$, then $x^{*}(0)=B$ and $x(t)\leq B$ when $t\neq 0$. Thus
\begin{align*}
\|v\|_{\mathcal{M}_{\varphi,\omega}^{o}}
&=\int_{R^{+}}x(t)v(t)dt\\
&\leq B\int_{supp~v}v(t)dt\\
&=B\int_{0}^{\mu(supp~v)}v^{*}(t)dt.
\end{align*}
Therefore $x\chi_{supp~v}(t)=B\chi_{supp~ v}$.\\
\textsl{Sufficiency}:
Assume $x(t)=B\chi_{supp~v}(t)+ x_{0}(t)$, therefore
\begin{equation*}
\int_{R^{+}}x(t)v(t)dt=B\int_{supp~v}v(t)dt=B\int_{0}^{\mu(supp~v)}v^{*}(t)dt=\|v\|_{\mathcal{M}_{\psi, \omega}^{o}}.
\end{equation*}
\end{proof}

Let $\sigma$ denote the measure preserving transformation such that $|x(t)|=x^{*}(\sigma(t))$. It is obvious that the results for $N$ function in \cite[Theorem 3.1]{Wang2024} is true for any Orlicz function.
\begin{theorem}\label{normattainluxemburg}For arbitrary Orlicz function $\varphi$ and decreasing weight $\omega$, $f\in \Lambda_{\varphi, \omega}^{*}$. Then
$f=L_{v}+s\:(0\neq v\in \mathcal{M}_{\psi, \omega}^{o}, s\in F)$ is norm attainable at $x\in S(\Lambda_{\varphi, \omega})$ if and only if\\
(1)$v(t)=v^{*}(\sigma(t))sign\: x(t)$ \\
(2)$\rho_{\varphi, \omega}(x)=1$.\\
(3)$s(x)=\|s\|$.\\
(4)$\int_{0}^{\infty}kv^{*}(t)x^{*}(t)dt=\rho_{\varphi, \omega}(x)+P_{\psi,\omega}(kv)$, $k\in K_{\mathcal{M}}(v)$.
\end{theorem}

\begin{lemma}\cite[Theorem 2.48]{Chen1996}
Let $x\in \Lambda_{\varphi, \omega}$, and $\theta(x)\neq 0$. Then there exist two singular functionals $s_{1}$ and $s_{2}$ such that $s_{1}\neq s_{2}$ and $s_{1}(x)=s_{2}(x)=\theta(x)$.
\end{lemma}

\begin{theorem}\cite[Theorem 3.3]{Wang2024}
Let $x\in \Lambda_{\varphi, \omega}$, then $Grad(x) \subset\mathcal{M}_{\psi, \omega}^{o}$ if and only if $\theta(x)<1$.
\end{theorem}

\begin{theorem}\cite[Theorem 3.4]{Wang2024}
Let
$\rho_{\varphi, \omega}(\frac{|x(t)|}{\|x\|})=1$. Then $v\in S(\mathcal{M}_{\varphi, \omega}^{o})$ is a supporting functional of $x$ if and only if $v=\frac{\varpi}{\|\varpi\|_{\mathcal{M}_{\psi, \omega}}^{o}}$ for some $\varpi$ satisfying
$$
p_{-}(\frac{x^{*}(t)}{\|x\|_{\varphi, \omega}})\omega(t)\leq \varpi^{*}(t)  \leq p(\frac{x^{*}(t)}{\|x\|_{\varphi, \omega}})\omega(t),\; a.e. \:on\: R_{+},
$$
i.e.,
$$
p_{-}(\frac{|x(t)|}{\|x\|_{\varphi, \omega}})\omega(\sigma(t))\leq \varpi(t)sign\: x(t)  \leq p(\frac{|x(t)|}{\|x\|_{\varphi, \omega}})\omega(\sigma(t)), a.e.\:on\:R_{+}.
$$
\end{theorem}
Then we come to $\Lambda_{\varphi, \omega}^{o}$. 

\begin{theorem}
For Orlicz function $\varphi$ such that $\lim_{t\rightarrow\infty}\frac{\varphi(t)}{t}=B<\infty$ and $\lim_{t\rightarrow\infty}Bt-\varphi(t)<\infty$. Let $x\in S(\Lambda_{\varphi, \omega}^{o})$ satisfying $\psi(B)\int_{0}^{\mu(supp\: x)}\omega(t)dt\leq 1$.
If $v\in \mathcal{M}_{\psi, \omega}$ is norm attainable at $x_{0}$, then\\
\begin{equation*}
\|v\|_{\mathcal{M}_{\psi, \omega}}=\frac{1}{B}\|v\|_{M_{W}}=\frac{1}{B}\sup_{\alpha\in R^{+}}\frac{1}{W(\alpha)}\int_{0}^{\alpha}v^{*}(t)dt.
\end{equation*}
\end{theorem}
\begin{proof}
\begin{align*}
\|v\|_{\mathcal{M}_{\psi, \omega}}
&=\sup\left\{\int_{R^{+}}x(t)v(t)dt: \|x\|_{\Lambda_{\varphi, \omega}^{o}}=1\right\}\\
&=\frac{1}{B}\sup\left\{\int_{R^{+}}Bx(t)v(t)dt: \|Bx\|_{\Lambda_{\omega}}=1\right\}\\
&=\frac{1}{B}\|v\|_{M_{W}}\\
&=\frac{1}{B}\sup\frac{1}{W(\alpha)}\int_{0}^{\alpha}v^{*}(t)dt.
\end{align*}
\end{proof}

\begin{theorem}\label{12}
For Orlicz function $\varphi$ such that $\lim_{t\rightarrow\infty}\frac{\varphi(t)}{t}=B<\infty$ and $\lim_{t\rightarrow\infty}Bt-\varphi(t)<\infty$. Let $x\in S(\Lambda_{\varphi, \omega}^{o})$ satisfying $\psi(B)\int_{0}^{\mu(supp\: x_{0})}\omega(t)dt\leq 1$.
$f=L_{v}+s$ is norm attainable at $x$ if and only if \\
(1) $s=0$.\\
(2) $\int_{R^{+}} x(t)v(t)dt=\int_{R^{+}} x^{*}(t)v^{*}(t)dt$.\\
(3) $v(t)\sim B\|f\|\omega(t)$.\\
\end{theorem}
\begin{proof}
\textsl{Necessity. }
Assume $f=L_{v}+s$ is norm-attainable at $x$.
Since $x\in E_{\varphi, \omega}^{o}$, then $s=0$.
If (2) is not satisfied, then $\int_{R^{+}}x(t)v(t)dt<\int_{R^{+}}x^{*}(t)v^{*}(t)dt$ and
\begin{align*}
\|f\|=f(x)
&=\int_{R^{+}}x(t)v(t)dt + s(x)\\
&=\int_{R^{+}}x(t)v(t)dt\\
&<    \int_{R^{+}}Bx^{*}(t)\frac{v^{*}(t)}{B}dt \\
&\leq \|Bx(t)\|_{\Lambda_{\omega}}\|\frac{v}{B}\|_{M_{W}}\\
&\leq \|v(t)\|_{\mathcal{M}_{\psi, \omega}}\\
&\leq \|f\|,
\end{align*}
a contradiction.

Since $\|v\|_{\mathcal{M}_{\psi, \omega}}=\|f\|-\|s\|=\|f\|$, then it is easy to verify $v\prec B\|f\| \omega$.
If $(3)$ is not satisfied, then $v^{*}\neq B\|f\|\omega(t)$ which implies $\sup_{a\in (0, m(supp\: x_{0})]}\frac{\int_{0}^{a}v^{*}(t)dt}{\int_{0}^{a}\omega(t)dt}< B \|f\| $.
Thus
\begin{align*}
\|f\|=f(x)
&=\int_{R^{+}}x(t)v(t)dt\\
&\leq \|Bx(t)\|_{\Lambda_{\omega}} \|\frac{v(t)}{B}\|_{M_{W}}\\
&< \|f\|.
\end{align*}
a contradiction.\\
\textsl{Sufficiency. }
Assume $f=L_{v}+s$ satisfying condition(1)(2)(3). Then
\begin{align*}
f(x)
&=\int_{R^{+}}v(t)x(t)dt \\
&=\int_{R^{+}}v^{*}(t)x^{*}(t)dt \\
&=B\|f\|\int_{R^{+}}x^{*}(t)\omega(t)dt \\
&=\|f\|.
\end{align*}
Then $f$ is norm-attainable at $x$.
\end{proof}

\begin{theorem}
$v\in Grad(x)$ if and only if $v(t)=B\omega(\sigma(t))$ 
\end{theorem}

\begin{proof}
Since $\int_{R^{+}}\omega(t)dt=\infty$, therefore $x_{0}\in E_{\varphi, \omega}^{o}$. For $f\in S(\Lambda_{\varphi, \omega}^{*})$, $f=L_{v}+s$. Then $s=0$. From theorem \ref{12} and the fact that $\int_{R^{+}}x(t)v(t)dt=\int_{R^{+}}x^{*}(t)v^{*}(t)dt$ implies $v(t)=v^{*}(\sigma(t))$(please refer to the proof process of \cite[Theorem 3.1]{Wang2024}), we have $v(t)=\omega(\sigma(t))$
\end{proof}
\begin{theorem}\label{NormattainOrlicznorm}\cite[Theore 3.1]{Wang2024}
Let $\varphi$ be an Orlicz function and $\omega$ be a decreasing weight.
$f=L_{v}+s(0\neq v\in \mathcal{M}_{\psi, \omega}, s\in F)$ is norm attainable at $x\in S(\Lambda_{\varphi, \omega}^{o})$ ($K(x)\neq \emptyset$) if and only if\\
(1)$v(t)=v^{*}(\sigma(t))sign x(t)$ \\
(2)$s(kx)=\|s\|$.\\
(3)$P_{\varphi, \omega}(\frac{v}{\|f\|})+\frac{\|s\|}{\|f\|}=1$.\\
(4)$\int_{0}^{\infty}\frac{kv(t)x(t)}{\|f\|}dt=\rho_{\varphi, \omega}(kx)+P_{\psi,\omega}(\frac{v}{\|f\|})$ where
$k\in K(x)$.\\
\end{theorem}

\begin{theorem}\label{th14}\cite[Theorem 3.7]{Wang2024}
For arbitrary Orlicz function $\varphi$ and decreasing weight $\omega$. Let $x\in S(\Lambda_{\varphi, \omega}^{o})$ and $K(x)\neq \emptyset$. Then $Grad(x)\subset \mathcal{M}_{\psi, \omega}$ 
if and only if one of the following conditions is satisfied.\\
	(1) $\theta(kx)<1$, $k\in K(x)$.\\
	(2) $P_{\psi,\omega}(p_{-}(kx^{*})\omega)=\rho_{\psi, \omega}(p_{-}(kx))=1$.
\end{theorem}

\begin{theorem}\cite[Theorem 3.8]{Wang2024}\label{Th15}
Assume $\varphi$ is an Orlicz function and $\omega$ is a decreasing weight. $L_{v}\in S(\mathcal{M}_{\psi, \omega})$ is a supporting functional of $x\in \Lambda_{\varphi, \omega}^{o}$($K(x)\neq \emptyset$) if and only if the following conditions are satisfied.\\
	(1) $P_{\psi, \omega}(v)=1$, $v(t)=v^{*}(\sigma(t))sign\:x(t)$. \\
	(2) $p_{-}(kx^{*}(t))\omega(t)\leq v^{*}(t)\leq p(kx^{*}(t))\omega(t)$ a.e in $R_{+}$ where $k\in K(x)$ and $\sigma$ is a measure preserving transformation such that $|x(t)|=x^{*}(\sigma(t))$, i.e.,
\begin{equation*}
p_{-}(k|x(t)|)\omega(\sigma(t))\leq v\:sign \:x(t)\leq p(k|x(t)|)\omega(\sigma(t)), a.e.\:on\:R_{+}.
\end{equation*}
\end{theorem}

\subsection{Exposed point in Orlicz-Lorentz space}\label{Exposed}
\begin{theorem}
Assume $\varphi$ is an Orlicz function with its complementary function $\psi$ satisfying $\lim_{t\rightarrow\infty}\frac{\psi(t)}{t}=B<\infty$ and $\lim_{t\rightarrow\infty} Bt-\psi(t)<\infty$. Let $v\in S(\mathcal{M}_{\psi, \omega})$ and $\varphi(B)\int_{0}^{\mu(supp~v)}\omega(t)dt<1$. Assume $v$ attain its norm at $x\in S(\Lambda_{\varphi, \omega})$, then $x(t)=B\chi_{supp~v}(t)+x_{0}(t)$ is not an exposed point of $B(\Lambda_{\varphi, \omega})$ where $(supp~v)\cap(supp~x_{0})=\emptyset$ and $x_{0}(t)\leq B$.
\end{theorem}

\begin{theorem}
Assume $\varphi$ is an Orlicz function with its complementary function $\psi$ satisfying $\lim_{t\rightarrow\infty}\frac{\psi(t)}{t}=B<\infty$ and $\lim_{t\rightarrow\infty} Bt-\psi(t)<\infty$. Let $v\in S(\mathcal{M}_{\psi, \omega})$ and $\varphi(B)\int_{0}^{\mu(supp~v)}\omega(t)dt=1$. Assume $v$ attain its norm at $x\in S(\Lambda_{\varphi, \omega})$, then $x(t)=B\chi_{supp~v}(t)$ is an exposed point of $B(\Lambda_{\varphi, \omega})$.
\end{theorem}

\begin{theorem}\label{Sprime}
For arbitrary $x\in\Lambda_{\varphi, \omega}$ and $x=x^{*}\circ \sigma$. If there exists $\alpha\in S^{\prime}$ such that $\mu\{t:x^{*}(t)=\alpha\}>0$. Then $x$ is not an exposed point of $B(\Lambda_{\varphi, \omega})$.
\end{theorem}
\begin{proof}
Since an exposed point is an extreme point, we can assume without loss of generality that $x(t)=s(t)+\alpha\chi_{A}$ where $\alpha\in S^{\prime}$ and $\mu(\sigma(A)\cap L(\omega))=0$. Select $\varepsilon>0$ such that $[\alpha-\varepsilon, \alpha+\varepsilon]\in S^{\prime}$ and divide $A$ into $A_{1}$ and $A_{2}$ such that
\begin{equation*}
\int_{\sigma(A_{1})}\omega(t)dt=\int_{\sigma(A_{2})}\omega(t)dt=\frac{1}{2}\int_{\sigma(A)}\omega(t)dt.
\end{equation*}
Assume $\mu A_{1}\geq \mu A_{2}$ without loss of generality. Define $x^{\prime}$ as follow
\begin{equation*}
x^{\prime}(t)=s(t)+(\alpha+\varepsilon)\chi_{A_{1}}+(\alpha-\varepsilon)\chi_{A_{2}}.
\end{equation*}
Therefore $x^{\prime}\circ\sigma^{-1}=(x^{\prime})^{*}$ and
\begin{equation*}
\rho_{\varphi, \omega}(x^{\prime})=\rho_{\varphi, \omega}(x)
\end{equation*}
Thus $x\in S(\Lambda_{\varphi, \omega})$.

If $\theta(x)<1$, then 
$Grad(x)\subset\mathcal{M}_{\psi, \omega}^{o}$. For $t$ such that $x^{*}(t)\in S$, then $p(x^{*}(t))=p_{-}(x^{*}(t))$. Let
\begin{equation*}
v(t)=p(x(t))\omega(\sigma(t))\chi_{A}+v^{\prime}(t)\chi_{R^{+}\backslash A}.
\end{equation*}
where $v^{\prime}(t)\in [p_{-}(x(t))\omega(\sigma(t)), p(x(t))\omega(\sigma(t))]$ for $t\in R^{+}\backslash A$.
Thus $\frac{v}{\|v\|_{\mathcal{M}_{\varphi, \omega}^{o}}}$ is the supporting functional of $x$ and $x^{\prime}$. Therefore $x$ is not an exposed point of $B(\Lambda_{\varphi, \omega})$.

If $\theta(x)=1$, from Hahn-Banach Theorem, there exists an singular functional $s\in F$ such that $s(x)=\|s\|$. Let $f=L_{v}+s$ where
$$
v(t)=p(x(t))\omega(\sigma(t))\chi_{A}+v^{\prime}(t)\chi_{R^{+}\backslash A}, v^{\prime}\in [p_{-}(x(t))\omega(\sigma(t)), p(x(t))\omega(\sigma(t))].
$$
In this case $\frac{f}{\|f\|}$ is the supporting functional of $x$ and $x^{\prime}$, therefore $x$ is not an exposed point of $B(\Lambda_{\varphi, \omega}^{o})$.
\end{proof}
\begin{theorem}
$x$
is an exposed point of $B(\Lambda_{\varphi, \omega})$ if and only if\\
(1)$\mu\{t\in R^{+}: x(t)\notin S \}=0$;\\
(2)$\rho_{\varphi, \omega}(x)=1$;\\
(3)$p_{-}(x^{*})\omega\in \mathcal{M}_{\psi, \omega}$\\
(4)Let $E_{1}=\{t\in R^{+}: x(t)\in A^{\prime}\}$, $E_{2}=\{t\in R^{+}:x(t)\in B^{\prime}\}$. Then $\mu E_{1}\cdot \mu E_{2}=0$
\end{theorem}

\begin{proof}
Necessity. Assume $x\in B(\Lambda_{\varphi, \omega})$ is an exposed point. Without loss of generality we can assume $x(t)\geq 0$.
From Theorem \ref{Sprime} condition (1) is necessary.
Since an exposed point is an extreme point, the necessity of (2) are obvious.

If condition (3) is not satisfied, then $x$ has only singular supporting functionals. Let $G(n)=\{t\in R^{+}: |x(t)|\geq n\}$, for singular functional $s\in Grad(x)$, we have $s(x\chi_{G(n)})=s(x)=\|s\|=1$. Therefore $x$ is not an exposed point of $B(\Lambda_{\varphi, \omega})$.

If condition (4) is not satisfied, there exist $a_{i}\in A^{\prime}$, $b_{j}\in B^{\prime}$ and $\varepsilon>0$ such that
$\mu G(a_{i})\neq 0$, $\mu G(b_{j})\neq 0$ where
\begin{align*}
G(a_{i})=\{t\in R^{+}: x^{*}(t)=a_{i}\},
G(b_{j})=\{t\in R^{+}: x^{*}(t)=b_{j}\}.\\
\end{align*}
Since $p$ is continuous at $a_{i}$ and $b_{j}$, there exist $\varepsilon$, $\delta$ such that
\begin{align*}
a_{i}+\varepsilon<b_{j}-\delta,  \;
p_{-}(a_{i}+\varepsilon)=p(a_{i}),\; p_{-}(b_{j}-\delta)=p(b_{j})
\end{align*}
and
\begin{equation*}
\int_{R^{+}\backslash G(a_{i})\cup G(b_{j})}\varphi(x^{*}(t))  \omega(t)dt +
\int_{G(a_{i})}\varphi(a_{i}+\varepsilon)\omega(t)dt+
\int_{G(b_{j})}\varphi(b_{j}-\delta)     \omega(t)dt
=1
\end{equation*}
Define $x^{\prime}\in\Lambda_{\varphi, \omega}$
\begin{equation*}
x^{\prime}(t)=x(t)\chi_{R^{+}\backslash \sigma^{-1}(G(a_{i})\cup G(b_{j}))}+ (a_{i}+\varepsilon)\chi_{\sigma^{-1}(G(a_{i}))}+(b_{j}-\delta)\chi_{\sigma^{-1}(G(b_{j}))}.
\end{equation*}
To calculate $\|x^{\prime}\|_{\varphi, \omega}$, we define
\begin{equation*}
x^{\prime \prime}(t)=x^{*}(t)\chi_{R^{+}\backslash (G(a_{i})\cup G(b_{j}))}+(a_{i}+\varepsilon)\chi_{G(a_{i})}+(b_{j}-\delta)\chi_{G(b_{j})}.
\end{equation*}
Since $x^{\prime \prime}$ is decreasing and $\mu_{x^{\prime}}(t)=\mu_{x^{\prime \prime}}(t)$, for $\lambda\in (\frac{a_{i}}{a_{i}+\varepsilon}, \frac{b_{j}}{b_{j}-\delta})$ we have
\begin{align*}
\rho_{\varphi, \omega}(\lambda x^{\prime})=\rho_{\varphi, \omega}(\lambda x^{\prime \prime})
&=\int_{R^{+}\backslash (G(a_{i})\cup G(b_{j}))}\varphi(\lambda x^{*})\omega(t)dt
+\int_{G(a_{i})}\varphi(\lambda(a_{i}+\varepsilon))\omega(t)dt\\
&\qquad+\int_{G(b_{j})}\varphi(\lambda(b_{j}-\delta))\omega(t)dt\\
&=\rho_{\varphi, \omega}(\lambda x).
\end{align*}
Thus $\|x^{\prime}\|_{\varphi, \omega}=\|x\|_{\varphi, \omega}=1$.

Let $\frac{f}{\|f\|^{o}}\in Grad(x)$ be the supporting functional of $x$ where $f=L_{v}+s$, we have
\begin{align*}
\|f\|^{o}=f(x)
&=\langle v,x  \rangle+ s(x)\\
&=\int_{R^{+}}v(t)x(t)dt+\|s\|\\
&=P_{\psi, \omega}(v)+\rho_{\varphi, \omega}(x)+\|s\|.
\end{align*}
Besides,
\begin{align*}
f(x^{\prime})
&=\langle v, x^{\prime}  \rangle+ s(x^{\prime})\\
&=\int_{R^{+}}v(t)y(t)dt+ s(x)+s(x^{\prime}-x)\\
&=P_{\psi, \omega}(v)+\rho_{\varphi, \omega}(x^{\prime})+\|s\|\\
&=P_{\psi, \omega}(v)+\rho_{\varphi, \omega}(x)+\|s\|\\
&=\|f\|^{o}.
\end{align*}
Thus $\frac{f}{\|f\|^{o}}$ is an exposed functional of $x$ and $x^{\prime}$ while $x\neq x^{\prime}$, a contradiction. Therefore $x$ is not an exposed point of $B(\Lambda_{\varphi, \omega})$.\\
\textsl{Sufficiency.} The sufficiency will be discussed in the following 4 cases.\\
\textsl{Case 1}: $P_{\psi, \omega}(p_{-}(x^{*})\omega)=1$ and $\theta(x)<1$. Then $L_{v}\in Grad(x)$ where $v=\frac{\varpi}{\|\varpi\|}$ and $\varpi=p_{-}(x(t))\omega(\sigma(t))$. Thus
\begin{align*}
\|p_{-}(x^{*}(t))\omega(t)\|_{\mathcal{M}_{\psi, \omega}^{o}}
&=\langle p_{-}(x(t))\omega(\sigma(t)), x(t)\rangle\\
&=\langle p_{-}(x^{*}(t))\omega(t), x^{*}(t) \rangle\\
&=\int_{R^{+}}(\psi(p_{-}(x^{*}(t)))+\varphi(x^{*}(t)))\omega(t)dt\\
&=\int_{R^{+}}\psi(\frac{p_{-}(x^{*}(t))\omega(t)}{\omega(t)})\omega(t)dt +\int_{R^{+}}\varphi(x^{*}(t))\omega(t)dt\\
&=1+P_{\psi, \omega}(v).
\end{align*}
If $v=\frac{p_{-}(x(t))\omega(\sigma(t))}{\|p_{-}(x^{*})\omega\|_{\mathcal{M}_{\psi, \omega}^{o}}}\in Grad(y)$,
 $y\in S(\Lambda_{\varphi, \omega})$. From Theorem \ref{normattainluxemburg}, $\rho_{\varphi, \omega}(y)=1$ and
\begin{align*}
\|p_{-}(x^{*}(t))\omega(t)\|_{\mathcal{M}_{\psi, \omega}^{o}}
&=\langle p_{-}(x(t))\omega(\sigma(t)), y(t) \rangle \\
&=\int_{R^{+}} y(t) p_{-}(x^{*}(\sigma(t)))\omega(\sigma(t))dt \\
&=\int_{R^{+}} y(\sigma^{-1}(t))p_{-}(x^{*}(t))\omega(t)dt \\
&\leq\int_{R^{+}}\left(\varphi(y(\sigma^{-1}(t)))+
\psi(\frac{p_{-}(x^{*}(t))\omega(t)}{\omega})\right)\omega(t)dt \\
&\leq\rho_{\varphi, \omega}(y)+P_{\psi, \omega}(v)\\
&=1+P_{\psi, \omega}(v).
\end{align*}
Thus
\begin{equation*}
p_{-}(y(\sigma^{-1}(t)))\leq p_{-}(x^{*}(t)) \leq p(y(\sigma^{-1}(t))),~a.e.~on~R^{+}.
\end{equation*}
Equivalently,
\begin{equation}\label{equal}
p_{-}(y)\leq p_{-}(x^{*}(\sigma(t)))=p_{-}(x(t))\leq p(y),~a.e.~on~R^{+}.
\end{equation}
Then $x=y$ for $t\in R^{+}$ such that $x^{*}(t)\in S$. For $t$ such that $x^{*}(t)\in A\cup B$, if $x(t)>y(t)$, Then $p_{-}(x)> P(y)$; if $x(t)< y(t)$, then $p_{-}(x)< P_{-}(y)$, and these lead to a contradiction to the inequality \ref{equal}. Thus $x(t)=y(t)$, $x(t)$ is an exposed point of $B(\Lambda_{\varphi, \omega})$.\\
\textsl{Case 2.} $P_{\psi, \omega}(p_{-}(x^{*})\omega)<1$ and $\theta(x)=1$.

In this case the supporting functionals of $x$ are in $\mathcal{M}_{\psi, \omega}^{o}\oplus F$ and have the form $f=L_{v}+s$. Let $\{r_{i}\}$ denote all of the discontinuous points of $p(t)$ and let $e_{i}=\{t\in R^{+}: x^{*}(t)=r_{i}\}$. Select $\varepsilon_{i}>0$ such that
\begin{equation*}
p_{-}(x^{*}(t))+\varepsilon_{i}\leq p(x^{*}(t)), \mu-a.e.~on~e_{i},~i=1,2...
\end{equation*}
and
\begin{equation*}
\int_{R^{+}\backslash \cup_{i}e_{i}}\psi(p_{-}(x^{*}(t)))\omega(t)dt+
\int_{\cup_{i}e_{i}}\psi(p_{-}(x^{*}(t)+\varepsilon_{i}))\omega(t)dt<1.
\end{equation*}
Let
\begin{equation*}
v(t)=p_{-}(x(t))\omega(\sigma(t))\chi_{R^{+}\backslash \sigma^{-1}(\cup_{i}e_{i})}+(p_{-}(x(t))
+\varepsilon_{i})\omega(\sigma(t))\chi_{\sigma^{-1}(\cup_{i}e_{i})}.
\end{equation*}
Thus $P_{\psi, \omega}(v)<1$. Since $\theta(x)\neq 0$, from Hahn-Banach Theorem, there exists a singular functional $s\in F$ such that $s(x)=\|s\|$. Let $\frac{f}{\|f\|^{o}}\in Grad(x)$ where $f=L_{v}+s$, thus
\begin{align*}
\|f\|^{o}=f(x)
&=\langle v, x \rangle +s(x)\\
&=P_{\psi, \omega}(v)+\rho_{\varphi, \omega}(x)+\|s\|\\
&=P_{\psi, \omega}(v)+1+\|s\|^{o}.
\end{align*}
If $\frac{f}{\|f\|^{o}}\in Grad(y)$, $y\in S(\Lambda_{\varphi, \omega})$. We have
\begin{align*}
\|f\|^{o}
&=\langle v, y \rangle + s(y)\\
&=\langle v(\sigma^{-1}(t)), y(\sigma^{-1}(t))\rangle + s(y)\\
&\leq\int_{R^{+}\backslash\cup_{i}e_{i}}\psi(p_{-}(x^{*}))
+\int_{\cup_{i}e_{i}}\psi(p_{-}(x^{*}(t)+\varepsilon_{i}))\omega(t)dt+\rho_{\varphi, \omega}(y)+\|s\|^{o}\\
&= P_{\psi, \omega}(v)+\rho_{\varphi, \omega}(y)+\|s\|^{o}\\
&= P_{\psi, \omega}(v)+1+\|s\|^{o}.
\end{align*}
Then 
\begin{align*}
p_{-}(y(\sigma^{-1}(t)))\leq p_{-}(x^{*}(t))\leq p(y(\sigma^{-1}(t))), \qquad                ~&a.e.~on~ R^{+}\backslash \cup_{i}e_{i}.\\
p_{-}(y(\sigma^{-1}(t)))\leq p_{-}(x^{*}(t))+\varepsilon_{i}\leq  p(y(\sigma^{-1}(t))), ~&a.e.~on ~\cup_{i}e_{i}.
\end{align*}
Using the same methods as Case 1 we can obtain $x(t)=y(t)$ a.e. on $R^{+}$, then $x$ is an exposed point of $B(\Lambda_{\varphi, \omega}$).\\
\textsl{Case 3}.
$P_{\psi, \omega}(p(x^{*})\omega)=1$, $\theta(x)<1$.
Thus $L_{v}\in Grad(x)$ where $v=\frac{\varpi}{\|\varpi\|}$ and $\varpi=p(x(t))\omega(\sigma(t))$. The proof is similar to Case 1.\\
\textsl{Case 4}.
$P_{\psi, \omega}(p_{-}(x^{*})\omega)<1< P_{\psi, \omega}(p(x^{*})\omega)$, $\theta(x)<1$.\\
In this case the singular functional $s=0$. Select $\varepsilon_{i}^{1}$ such that
$p_{-}(x^{*}(t))+\varepsilon_{i}^{1}\leq p(x^{*}(t))$ for $t\in e_{i}$ and
\begin{equation*}
\int_{R^{+}\backslash \cup_{i}e_{i}}\psi(p_{-}(x^{*}(t)))\omega(t)dt+
\int_{\cup_{i}e_{i}}\psi(p_{-}(x^{*}(t)+\varepsilon_{i}^{1}))\omega(t)dt=1.
\end{equation*}
Let
\begin{equation*}
v(t)=p_{-}(x(t))\omega(\sigma(t))\chi_{R^{+}\backslash \sigma^{-1}(\cup_{i}e_{i})}+(p_{-}(x(t))+\varepsilon_{i}^{1})\omega(\sigma(t))\chi_{\sigma^{-1}(\cup_{i}e_{i})}.
\end{equation*}
In this case $\frac{v}{\|v\|_{\mathcal{M}_{\psi, \omega}^{o}}}\in Grad(x)$.
Thus
\begin{align*}
\|v\|_{\mathcal{M}_{\psi, \omega}^{o}}
&=\langle v,x \rangle\\
&=\langle v(\sigma^{-1}(t)), x(\sigma^{-1}(t)) \rangle\\
&=\int_{R^{+}\backslash \cup_{i}e_{i}}\psi(p_{-}(x^{*}(t)))\omega(t)dt+\int_{\cup_{i}e_{i}}\psi(p_{-}(x^{*}(t))+\varepsilon_{i}^{1})\omega(t)dt
+\int_{R^{+}}\varphi(x^{*}(t))\omega(t)dt\\
&=P_{\psi, \omega}(v)+1.
\end{align*}
If $\frac{v}{\|v\|_{\mathcal{M}_{\psi, \omega}^{o}}}$ is the supporting functional of $y\in S(\Lambda_{\varphi, \omega})$, we have
\begin{align*}
\|v\|_{\mathcal{M}_{\psi, \omega}^{o}}
&=\langle v, y                                 \rangle\\
&=\langle v(\sigma^{-1}(t)), y(\sigma^{-1}(t)) \rangle\\
&=\int_{R^{+}}v(\sigma^{-1}(t))y(\sigma^{-1}(t))dt\\
&=\int_{R^{+}\backslash (\cup_{i}e_{i})}\psi(p_{-}(x^{*}))\omega(t)dt+\int_{\cup_{i}e_{i}}\psi(p_{-}(x^{*})+\varepsilon_{i}^{1})+\int_{\varphi}(y(\sigma^{-1}))\omega(t)dt\\
&\leq P_{\psi, \omega}(v)+\rho_{\varphi, \omega}(y)\\
&\leq P_{\psi, \omega}(v)+1\\
&= \|v\|_{\mathcal{M}_{\psi, \omega}^{o}}.
\end{align*}
Thus
\begin{equation*}
p_{-}(y(\sigma^{-1}(t)))\leq p_{-}(x^{*}(t))\leq p(y(\sigma^{-1}(t))),~t\in R^{+}\backslash (\cup_{i}e_{i})
\end{equation*}
and
\begin{equation*}
p_{-}(y(\sigma^{-1}(t)))\leq p_{-}(x^{*}(t))+\varepsilon_{i}^{1}\leq p(y(\sigma^{-1}(t))), ~t\in \cup_{i}e_{i}.
\end{equation*}
Then we can obtain
$$
x(t)=y(t), ~a.e.~on ~R^{+}.
$$
and $x$ is an exposed point of $B(\Lambda_{\varphi, \omega})$.
\end{proof}
Then we will discuss the exposed point in Orlicz-Lorentz space quipped with the Orlicz norm.


\begin{theorem}\label{notexposedpointOrlicznorm}
$x\in S(\Lambda_{\varphi, \omega}^{o})$, $x(t)=x^{*}(\sigma(t))$ and $k\in K(x)$. If $x=\alpha\chi_{A}$ where $k\alpha\in S^{\prime}$ and $\mu(\sigma(A)\cap L(\omega))=0$. Then $x$ is not an exposed point of $B(\Lambda_{\varphi, \omega})$.
\end{theorem}
\begin{proof}
By the similar method in theorem \ref{Sprime}, there exists $x^{\prime}$ defined by
\begin{equation*}
kx^{\prime}(t)= x(t)\chi_{R^{+}\backslash  A}+ (\alpha+\varepsilon)\chi_{A_{1}}+(\alpha-\varepsilon)\chi_{A_{2}}.
\end{equation*}
Therefore $\rho_{\psi, \omega}(p(kx^{\prime}))=\rho_{\psi, \omega}(p(kx))=1$ and $\|x^{\prime}\|_{\varphi, \omega}^{o}=\frac{1}{k}\left( 1+\rho_{\varphi, \omega}(kx^{\prime}) \right)=1$. If $f$ is the supporting functional of $x$, then $f$ is the supporting functional of $x^{\prime}$ as well. Therefore $x$ is not an exposed point of $B(\Lambda_{\varphi, \omega}^{o})$.
\end{proof}

\begin{theorem} Let $\varphi$ be an Orlicz function and $\omega$ be a decreasing weight.
Assume $x(t)=x^{*}(\sigma(t))\in S(\Lambda_{\varphi, \omega}^{o})$ and $K(x)\neq \emptyset$. $x$ is an exposed point of $B(\Lambda_{\varphi, \omega}^{o})$ if and only if \\
(1) $K(x)=\{ k\}$. \\
(2)$\mu\{t\in R^{+}: kx(t)\notin S \}=0$; \\
(3)$\mu E=\{t\in R^{+}: kx^{*}(t)\in A^{\prime}\cup B^{\prime}\}=0$.\\
(4)
If $P_{\psi, \omega}(p_{-}(k(x^{*}))\omega)=1$, then $\mu\{t\in R^{+}: kx^{*}(t)\in B\}=0$.
If $\:\theta(kx)<1$ and $P_{\psi, \omega}(p(kx^{*})\omega)=1$,
then $\mu\{t\in R^{+}: kx^{*}(t)\in A\}=0$.
\end{theorem}
\begin{proof}
\textsl{Necessity.}
Assume $x\in B(\Lambda^{o}_{\varphi, \omega})$ is an exposed point of $B(\Lambda_{\varphi, \omega}^{o})$. Without loss of generality, we can suppose $x(t)\geq 0$.
Since an exposed point is an extreme point, condition (1) is true.
From Theorem \ref{notexposedpointOrlicznorm}, condition (2) is true.
If condition (3) is not satisfied, It can be discussed in three cases:    \\
(a)
$\mu\{ t\in R^{+}:kx^{*}(t)\in A^{\prime}\cup B^{\prime}\}\neq 0$, $\mu\{t\in R^{+}: kx^{*}(t)\in A^{\prime} \}=0$. \\
(b)
$\mu\{ t\in R^{+}: kx^{*}(t)\in A^{\prime}\cup B^{\prime}\}\neq 0$, $\mu \{t\in R^{+}: kx^{*}(t)\in B^{\prime}\}=0$. \\
(c)
$\mu\{ t\in R^{+}: kx^{*}(t)\in A^{\prime}\}\neq 0$, $\mu \{t\in R^{+}: kx^{*}(t)= B^{\prime}\}\neq 0$.\\

We will only discuss case (a). There exists $b_{i}\in B^{\prime}$ such that $G(b_{i})=\{t\in R^{+}: kx^{*}(t)= b_{i}\}$ and $\mu G(b_{i})>0$.
select $\varepsilon>0$ such that
\begin{equation}\label{dir}
p_{-}(b_{i}-\varepsilon)=p(b_{i}-\varepsilon)=p_{-}(b_{i})=p(b_{i}),
\end{equation}
Define $x^{\prime}\in \Lambda_{\varphi, \omega}^{o}$ such that
$$
kx^{\prime}=kx(t)\chi_{R^{+}\backslash \sigma^{-1}(G(b_{i}))}+ (b_{i}-\varepsilon)\chi_{\sigma^{-1}(G(b_{i}))}.
$$
Thus $x^{\prime}\leq x$. Consider $x^{\prime \prime}$ defined as
\begin{align*}
kx^{\prime \prime}
&=kx^{*}\chi_{R^{+}\backslash G(b_{i})}+(kx^{*}-\varepsilon)\chi_{G(b_{i})}\\
&=kx^{*}\chi_{R^{+}\backslash G(b_{i})}+(b_{i} -\varepsilon)\chi_{G(b_{i})}.
\end{align*}
Since
\begin{align*}
\mu_{kx^{\prime}}(\lambda)
&=\mu\{t\in R^{+}: |kx^{\prime}|>\lambda\}\\
&=\mu\{t\in R^{+}\backslash \sigma^{-1}(G(b_{i})):|kx|>\lambda\}+\mu \{t\in \sigma^{-1}(G(b_{i})):  |b_{i}-\varepsilon|>\lambda\}\\
&=\mu\{t\in R^{+}\backslash G(b_{i}):|kx^{*}|>\lambda\} +\mu \{t\in G(b_{i}):  |b_{i}-\varepsilon|>\lambda\}\\
&=\mu_{kx^{\prime \prime}}(\lambda).
\end{align*}
Therefore from equation (\ref{dir}) we have
\begin{align*}
\rho_{\psi, \omega}(p(kx^{\prime}))
&=\int_{R^{+}}\psi(p(k(x^{\prime})^{*}))\omega(t)dt\\
&= \int_{R^{+}}\psi(p(k(x^{\prime \prime})^{*}))\omega(t)dt\\
&\geq \int_{R^{+}\backslash G(b_{i})}\psi(p(kx^{*}))\omega(t)dt+ \int_{G(b_{i})}\psi(p(kx^{*}-\varepsilon))\omega(t)dt\\
&\geq \int_{R^{+}\backslash G(b_{i})}\psi(p(kx^{*}))\omega(t)dt+ \int_{G(b_{i})}\psi(p(b_{i}))\omega(t)dt\\
&=\rho_{\psi, \omega}(p(kx))\\
&= 1.
\end{align*}
Since $x^{\prime}\leq x$, we have $(x^{\prime})^{*}\leq x^{*}$ and
\begin{align*}
1\leq \rho_{\psi, \omega}(p(kx^{\prime}))\leq\rho_{\psi,\omega}(p(kx))=1.
\end{align*}
We have $\rho_{\psi, \omega}(p(kx^{\prime}))=1$ and $k\in K(x^{\prime})$.
Let $f=L_{v}+s$, where $v=p(kx^{*}(\sigma(t)))\omega(\sigma(t))$, and $\frac{f}{\|f\|}\in Grad(x)$. Then $f$ is norm attainable at $x$.
Form Theorem \ref{NormattainOrlicznorm} we have
\begin{align*}
\|f\|=f(x)
&=\frac{1}{k}\left(\int_{R^{+}}kx(t)p(kx)\omega(\sigma(t))dt +s(kx) \right)\\
&=\frac{1}{k}\left( \int_{R^{+}}kx^{*}(t)p(kx^{*}(t))\omega(t)dt+\|s\|\right)\\
&=\frac{1}{k}\left(P_{\psi, \omega}(p(kx^{*})\omega)+\rho_{\varphi, \omega}(kx)+\|s\|\right)\\
&=\frac{1}{k}\left(\rho_{\psi, \omega}(p(kx))+\rho_{\varphi, \omega}(kx)+\|s\|\right)\\
&=\frac{1}{k}\left(1+\rho_{\varphi, \omega}(kx)+\|s\|\right)\\
&=1+\frac{1}{k}\|s\|.
\end{align*}
Then we will consider $f(x^{\prime})$.
First we will discuss the following integral.
\begin{align*}
\int_{R^{+}}x^{\prime}p(kx^{*}(\sigma(t)))\omega(\sigma(t))dt.
\end{align*}
Define $F(t)=kx^{\prime}(t)p(kx^{*}(\sigma(t)))\omega(\sigma(t))$.
Since
\begin{align*}
kx^{\prime}(\sigma^{-1}(t))
&=(kx-\varepsilon \chi_{\sigma^{-1}(G(b_{i}))})(\sigma^{-1}(t))\\
&=kx(\sigma^{-1}(t))-\varepsilon\chi_{\sigma^{-1}(G(b _{i}))}(\sigma^{-1}(t))\\
&=kx(\sigma^{-1}(t))-\varepsilon\chi_{G(b_{i})}\\
&=kx^{*}(t)-\varepsilon\chi_{G(b_{i})}.
\end{align*}
Therefore we have
\begin{align*}
F(\sigma^{-1}(t))=kx^{\prime}(\sigma^{-1}(t))(t)p(kx^{*}(t))\omega(t)=(kx^{*}(t)-\varepsilon\chi_{\sigma(E)})p(kx^{*}(t))\omega(t).
\end{align*}
it follows that
\begin{align*}
\int_{R^{+}}F(t)dt =\int_{R^{+}}F(\sigma^{-1}(t))dt.
\end{align*}
Since $\varphi(x^{\prime})\chi_{\sigma^{-1}(G(b_{i}))}$ and $\varphi(x^{\prime\prime}(\sigma(t)))\chi_{\sigma^{-1}(G(b_{i}))}$ have the same distribution function,
we have
\begin{align*}
\|f\|
&\geq f(\frac{x^{\prime}}{\|x^{\prime}\|_{\varphi, \omega}^{o}})=\frac{1}{\|x^{\prime}\|_{\varphi, \omega}^{o}}\left(\int_{R^{+}}x^{\prime}p(kx^{*}(\sigma(t)))\omega(\sigma(t))dt\right)
+\frac{1}{k\|x^{\prime}\|_{\varphi, \omega}^{o}}s(kx^{\prime})\\
&=\frac{1}{k\|x^{\prime}\|_{\varphi, \omega}^{o}} \left( \int_{R^{+}}kx^{\prime}p(kx^{*}(\sigma(t)))\omega(\sigma(t))dt\right)
+\frac{1}{k\|x^{\prime}\|_{\varphi, \omega}^{o}}(s(kx)+s(kx^{\prime}-kx))\\
&=\frac{1}{k\|x^{\prime}\|_{\varphi, \omega}^{o}}\left( \int_{R^{+}}(kx^{*}-\varepsilon\chi_{G(b_{i})})p(x^{*})\omega(t)dt\right)+\frac{1}{k\|x^{\prime}\|_{\varphi, \omega}^{o}}s(kx)\\
&=\frac{1}{k\|x^{\prime}\|_{\varphi, \omega}^{o}}\left( \int_{R^{+}}\varphi(kx^{*}-\varepsilon\chi_{G(b_{i})})\omega(t)dt
+\int_{R^{+}}\psi(p(kx^{*}))\omega(t)dt\right)+\frac{1}{k\|x\|_{\varphi, \omega}^{o}}\|s\|\\
&=\frac{1}{k\|x^{\prime}\|_{\varphi, \omega}^{o}}\left(\int_{R^{+}}\varphi(kx^{*}-\varepsilon\chi_{G(b_{i})})\omega(t)dt
+1\right)+\frac{1}{k\|x^{\prime}\|_{\varphi, \omega}^{o}}\|s\|\\
&=\frac{1}{k\|x^{\prime}\|_{\varphi, \omega}^{o}}
\left(\int_{R^{+}\backslash G(b_{i})}\varphi((kx^{\prime})^{*})\omega(t)dt
+\int_{G(b_{i})}\varphi(x^{\prime\prime})\omega(t)dt+1\right)
+\frac{1}{k\|x^{\prime}\|_{\varphi, \omega}^{o}}\|s\|\\
&=\frac{1}{k\|x^{\prime}\|_{\varphi, \omega}^{o}}\left(\int_{R^{+}\backslash G(b_{i})}\varphi(k(x^{\prime})^{*})\omega(t)dt+
\int_{G(b_{i})} \varphi((x^{\prime})^{*})\omega(t)dt+1 \right)+\frac{1}{k\|x^{\prime}\|_{\varphi, \omega}^{o}}\|s\|\\
&=\frac{1}{k\|x^{\prime}\|_{\varphi, \omega}^{o}}\left(\int_{R^{+}}\varphi(k(x^{\prime})^{*})\omega(t)dt+1\right)
+\frac{1}{k\|x^{\prime}\|_{\varphi, \omega}^{o}}\|s\|\\
&=\frac{1}{k\|x^{\prime}\|_{\varphi, \omega}^{o}}
\left(\rho_{\varphi, \omega}(k\|x^{\prime}\|_{\varphi, \omega}^{o}
\frac{x^{\prime}}{\|x^{\prime}\|_{\varphi, \omega}^{o}})+1\right)+\frac{1}{k\|x^{\prime}\|_{\varphi, \omega}^{o}}\|s\|\\
&\geq 1+\frac{\|s\|}{k}\\
&=\|f\|.
\end{align*}
Therefore $\frac{f}{\|f\|}\in Grad(x)$ and $\frac{f}{\|f\|}\in Grad(\frac{x^{\prime}}{\|x^{\prime}\|_{\varphi,\omega}^{o}})$,
which implies that $x$ is not an exposed points of $B(\Lambda_{\varphi, \omega}^{o})$, a contradiction.


If condition (4) is not necessary, we consider the following two cases. \\
\textsl{Case 1}: $P_{\psi, \omega}(p_{-}(kx^{*})\omega)=1$ and there exists $b_{j}\in B$ and $\varepsilon>0$ such that $p_{-}(b_{j}-\varepsilon)=p(b_{j}-\varepsilon)= p_{-}(b_{j})$, $\mu G(b_{j})= \mu\{t\in R^{+}: kx^{*}(t)=b_{j}\}>0$.
Define $x^{\prime}$ such that
$$
kx^{\prime}= kx-\varepsilon\chi_{\sigma^{-1}(G(b_{j}))}.
$$
Following the way of condition (3), we have
$\|p_{-}(kx(t))\omega(\sigma(t))\|_{\mathcal{M}_{\psi, \omega}}=1$ and $p_{-}(kx(t))\omega(\sigma(t))\in Grad(x)$. It is easy to verified that $p_{-}(kx(t))\omega(\sigma(t))$ is also the supporting functional of $\frac{x^{\prime}}{\|x^{\prime}\|_{\varphi, \omega}^{o}}$, and it leads to a contradiction.\\
\textsl{Case 2}. If $\theta(kx)<1$, $P_{\psi, \omega}(p(kx^{*}(t))\omega(t))=1$, there exists $a_{i}\in A$ and $\varepsilon>0$, such that $p(a_{i}-\varepsilon)= p(a_{i})$, $\mu G(a_{i})=\mu \{t\in R^{+}: kx(t)=a_{i}\}>0$.
Define
$$
kx^{\prime}=kx +\varepsilon\chi_{\sigma^{-1}(G(a_{i}))}.
$$
Since $\rho_{\psi, \omega}(p(kx^{\prime}))=\rho_{\psi, \omega}(p(kx))=1$,
following the method in proof of condition (3),  $p(k(x^{*}(\sigma(t))))\omega(\sigma(t))$ is the supporting functional of $x$ and $\frac{x^{\prime}}{\|x\|_{\varphi, \omega}^{o}}$. Thus $x$ is not an exposed point of $B(\Lambda^{o}_{\varphi, \omega})$, a contradiction.\\
~\\
\textsl{Sufficiency: }We will prove the sufficiency in the following cases. \\
\textsl{Case 1.}\;$P_{\psi, \omega}(p_{-}(kx^{*}(t))\omega(t))=1$ and $\mu \{t\in R^{+}: |kx^{*}(t)|\in B \}=0$.\\
By Theorem \ref{th14} and Theorem \ref{Th15}, $v(t)=p_{-}(kx^{*}(\sigma(t)))\omega(\sigma(t))$ is a supporting functional of $x$. If $v$ is a supporting functional of $y\in S(\Lambda_{\varphi, \omega}^{o})$, let $k_{y}\in K(y)$ we have
\begin{align*}
\|y\|_{\varphi,\omega}^{o}=\langle y, v\rangle
&=\frac{1}{k_{y}}\int_{R^{+}}k_{y}y(t)p_{-}(kx^{*}(\sigma(t)))\omega(\sigma(t))dt\\
&=\frac{1}{k_{y}}\int_{R^{+}}(k_{y}y(\sigma^{-1}(t)))p_{-}(kx^{*}(t))\omega(t)dt\\
&\leq\frac{1}{k_{y}}\int_{R^{+}}\varphi(k_{y}y(\sigma^{-1}(t)))\omega(t)+\psi\left(\frac{p_{-}(kx^{*})\omega}{\omega}\right)\omega(t)\\
&\leq\frac{1}{k_{y}}\int_{R^{+}}\varphi(k_{y}y^{*}(t))\omega(t)dt+\psi\left(\frac{p_{-}(kx^{*})\omega}{\omega}\right)\omega(t)\\
&=\rho_{\varphi, \omega}(k_{y}y)+P_{\psi, \omega}(p(kx^{*}(t))\omega)\\
&=\frac{1}{k_{y}}(\rho_{\varphi, \omega}(k_{y}y)+1)\\
&=\|y\|_{\varphi, \omega}^{o}.
\end{align*}
Therefore we have
\begin{equation*}
p_{-}(k_{y}y(\sigma^{-1}(t)))\leq p_{-}(kx^{*}(t)) \leq p(k_{y}y(\sigma^{-1}(t))),
\end{equation*}
and
\begin{equation*}
p_{-}(k_{y}y)\leq p_{-}(kx^{*}(\sigma(t)))=p_{-}(kx)\leq p(k_{y}y).
\end{equation*}
Obviously, $k_{y}y=kx$ for $t\in R^{+}$ such that $x^{*}(t)\in S$ or $x^{*}(t)\in A$.
Therefore $k_{y}=k_{y}\|y\|_{\varphi, \omega}^{o}=\|k_{y}y\|_{\varphi, \omega}^{o}=\|kx\|_{\varphi, \omega}=k$ and $x=y$, which implies $x$ is an exposed point. \\
\textsl{Case 2.}\:$P_{\psi, \omega}(p_{-}(kx^{*})\omega)<1$ and $\theta(kx)=1$.\\
Since $\theta(kx)=1$, the supporting functional of $x$ is in $\mathcal{M}_{\psi, \omega}\oplus F$.
Select $\varepsilon_{i}^{2}>0$ such that
$$
p_{-}(kx^{*}(t))+\varepsilon^{2}_{i}\leq p(kx^{*}(t)), \;a.e.\: on\: e_{i},\; i=1,2,...n,
$$
and
$$
\int_{R^{+}\backslash \cup_{i}e_{i}}\psi(p_{-}(kx^{*}(t)))\omega(t)dt+\int_{\cup_{i}e_{i}}\psi(p_{-}
(kx^{*}(t))+\varepsilon_{i})dt < 1
$$
as well as
$p_{-}(kx^{*}(t))\chi_{R^{+}\backslash\cup_{i}e_{i}}+(p_{-}(kx^{*}(t))+\varepsilon_{i})\chi_{\cup_{i}e_{i}}$
is decreasing. Let
\begin{equation*}
v(t)=p_{-}(kx^{*}(\sigma(t)))\omega(\sigma(t))\chi_{R^{+}\backslash \sigma^{-1}(\cup_{i}e_{i})}+(p_{-}(kx^{*}(\sigma(t)))+\varepsilon^{2}_{i})\omega(\sigma(t)))\chi_{\sigma^{-1}(\cup_{i}e_{i})}.
\end{equation*}
Thus $P_{\psi, \omega}(v)<1$.
If there is $y\in S(\Lambda_{\varphi, \omega}^{o})$ such that $f\in Grad(y)$ as well, assume $k_{y}\in K(y)$,
\begin{align*}
1=\|f\|_{\mathcal{M}_{\varphi, \omega}^{o}}
&=\langle v, y\rangle +s(y)\\
&=\frac{1}{k_{y}}\left(\langle v(\sigma^{-1}(t), y(\sigma^{-1}(t)))\rangle +s(k_{y}y)\right)\\
&=\frac{1}{k_{y}}\left(\int_{R^{+}}\frac{v(\sigma^{-1}(t))}{\omega(t)}y(\sigma^{-1}(t))\omega(t)dt+s(k_{y}y)\right)\\
&\leq \frac{1}{k_{y}}\left(\int_{R^{+}}\psi(\frac{v(\sigma^{-1}(t))}{\omega(t)})\omega(t)dt +\int_{R^{+}}\varphi(y(\sigma^{-1}(t)))+s(k_{y}y)\right)\\
&=\frac{1}{k_{y}}\left( P_{\psi, \omega}(v)+\rho_{\varphi, \omega}(k_{y}y(\sigma^{-1}(t)))+\|s\|\right)\\
&\leq \frac{1}{k_{y}}(1+\rho_{\varphi, \omega}(k_{y}y))\\
&=1.
\end{align*}
Thus we have $\|s\|=s(k_{y}y)=s(kx)$. 
Besides,
\begin{equation*}
p_{-}(k_{y}y(\sigma^{-1}(t)))\leq p_{-}(kx^{*}(t))\leq p(k_{y}y(\sigma^{-1}(t))),\: t\in R^{+}\backslash (\cup_{i}e_{i}),
\end{equation*}
and
\begin{equation*}
p_{-}(k_{y}y(\sigma^{-1}(t)))\leq p_{-}(kx^{*}(t))+\varepsilon_{i}\leq p(k_{y}y(\sigma^{-1}(t))),  \: t\in \cup_{i}e_{i}.
\end{equation*}
This implies
$$
kx(t)=kx^{*}(\sigma(t))=k_{y}y(t), \sigma(t)\in S\; or \; \sigma(t)\in e_{i},\;i=1,2...n.
$$
By $k_{y}=\|k_{y}y\|_{\varphi, \omega}^{o}=\|kx\|_{\varphi, \omega}^{o}=k$, we have $y=x$. It follows that $x$ is an exposed point of $B(\Lambda_{\varphi, \omega}^{o})$. \\
\textsl{Case 3.} $P_{\psi, \omega}(p(kx^{*})\omega)=1$ and $\mu\{t\in R^{+}:\: kx^{*}(t)\in A\}=0$.\\
The proof is similar to Case 1. \\
\textsl{Case 4.} $P_{\psi, \omega}(p_{-}(kx^{*})\omega)<1<P_{\psi, \omega}(p(kx^{*})\omega)$ and $\theta(kx)<1$. \\
We have $Grad(x)\subset\mathcal{M}_{\psi, \omega}$ and we can select $\varepsilon_{i}^{3}>0$ such that
$$
p_{-}(kx^{*}(t))+\varepsilon^{3}_{i}\leq p(kx^{*}(t)), \;a.e.\: on\: e_{i},\; i=1,2,...n,
$$
and
$$
\int_{R^{+}\backslash \cup_{i}e_{i}}\psi(kx^{*}(t))\omega(t)dt+\int_{\cup_{i}e_{i}}\psi(kx^{*}(t)+\varepsilon^{3}_{i})\omega(\sigma(t))dt = 1.
$$
Let
\begin{equation*}
v(t)=p_{-}(kx^{*}(\sigma(t)))\omega(\sigma(t))\chi_{R^{+}\backslash \cup_{i}\sigma^{-1}(e_{i})}+(p_{-}(kx^{*}(\sigma(t)))
+\varepsilon^{3}_{i})\omega(\sigma(t)))\chi_{\cup_{i}\sigma^{-1}(e_{i})}.
\end{equation*}
Thus $P_{\psi, \omega}(v)=1$. If $in Grad(y)$, $y\in S(\Lambda_{\varphi, \omega}^{o})$. Then
\begin{align*}
\|y\|_{\varphi, \omega}^{o}
&=\frac{1}{k_{y}}\langle v, y\rangle \\
&=\frac{1}{k_{y}}\langle v(\sigma^{-1}(t)), y(\sigma^{-1}(t))\rangle\\
&=\frac{1}{k_{y}}\int_{R^{+}}\frac{v(\sigma^{-1}(t))}{\omega(t)}y(\sigma^{-1}(t))\omega(t)dt\\
&\leq\frac{1}{k_{y}}\left(\int_{R^{+}}\psi(\frac{v(\sigma^{-1}(t))}{\omega(t)})\omega(t)dt+ \int_{R^{+}}\varphi(y(\sigma^{-1}(t)))\omega(t)dt\right)\\
&\leq\frac{1}{k_{y}}(1+\rho_{\varphi, \omega}(k_{y}y))\\
&=\|y\|_{\varphi, \omega}^{o}.
\end{align*}
By the same method in Case 2, we can obtain
\begin{equation*}
k_{y}y(t)= kx(t), \sigma(t)\in S \; or\;\sigma(t)\in \cup_{i}e_{i},\; i=1,2...n.
\end{equation*}
Besides, $k_{y}=k_{y}\|y\|=kx=k$, and it implies $x=y$, $x$ is an exposed point of $B(\Lambda_{\varphi, \omega}^{o})$.
\end{proof}

\nocite{Halperin195305}\nocite{Ryff1970449}
\section{Acknowledgement}

The authors gratelfully acknowledge financial support from China Scholarship Council and the National Natural Science Foundation of P. R. China (Nos. 11971493).



\bigskip
%
%
%
%

\begin{appendices}




\end{appendices}
\end{document}